\documentclass[11pt,a4paper,reqno]{amsart}
\usepackage{amsmath, amssymb, latexsym, enumerate,  graphicx,tikz, stmaryrd, eufrak,enumitem}

\usepackage{thmtools}

\usepackage[pdftex]{hyperref}
\usepackage{cleveref}
\usepackage{bbm}
\usepackage{cases}
\usepackage{array}
\usepackage{multirow}
\usepackage{tikz-cd}
\usepackage[all]{xy}

\usepackage{mathtools}

\usepackage[hmarginratio={1:1},vmarginratio={1:1},lmargin=80.0pt,tmargin=90.0pt]{geometry}

\usepackage{dynkin-diagrams}

\usepackage{tikz}
\tikzset{anchorbase/.style={baseline={([yshift=-0.5ex]current bounding box.center)}}}
\usetikzlibrary{decorations.markings}
\usetikzlibrary{decorations.pathreplacing}
\usetikzlibrary{arrows,shapes,positioning,backgrounds}
\tikzstyle directed=[postaction={decorate,decoration={markings,
    mark=at position #1 with {\arrow{>}}}}]
\tikzstyle rdirected=[postaction={decorate,decoration={markings,
    mark=at position #1 with {\arrow{<}}}}]
\tikzset{double line with arrow/.style args={#1,#2}{decorate,decoration={markings,%
mark=at position 0 with {\coordinate (ta-base-1) at (0,1pt);
\coordinate (ta-base-2) at (0,-1pt);},
mark=at position 1 with {\draw[#1] (ta-base-1) -- (0,1pt);
\draw[#2] (ta-base-2) -- (0,-1pt);
}}}}
\tikzset{Equal/.style={-,double line with arrow={-,-}}}
    
\usepackage{graphicx}
\usepackage{nicefrac}

 \newlength{\baseunit}               
 \newcount{\numlines}                
\newtheorem{thm}{Theorem}
\newtheorem{theorem}[subsection]{Theorem}
\newtheorem{lemma}[theorem]{Lemma}
\newtheorem{prop}[theorem]{Proposition}

\theoremstyle{definition}

\newtheorem{remark}[theorem]{Remark}

\newtheorem{example}[subsection]{Example}

\newtheorem{question}[theorem]{Question}

\newcommand{\mG}{\mathbb{G}}

\newcommand{\bk}{\mathbf{k}}

\newcommand{\cA}{\mathcal{A}}
\newcommand{\cB}{\mathcal{B}}
\newcommand{\cC}{\mathcal{C}}
\newcommand{\cD}{\mathcal{D}}

\numberwithin{equation}{section}

\newcommand{\cM}{\mathcal{M}}

\newcommand{\cT}{\mathcal{T}}
\newcommand{\cI}{\mathcal{I}}
\newcommand{\cN}{\mathcal{N}}

\newcommand{\ev}{\mathrm{ev}}
\newcommand{\coev}{\mathrm{coev}}
\newcommand{\im}{\mathrm{im}}

\newcommand{\ch}{\mathrm{ch}}

\newcommand{\Spec}{\mathrm{Spec}}

\newcommand{\Ver}{\mathsf{Ver}}

\newcommand{\Comod}{\mathsf{Comod}}
\newcommand{\Rep}{\mathsf{Rep}}

\newcommand{\unit}{{\mathbbm{1}}}

\newcommand{\cO}{\mathcal{O}}

\newcommand{\mN}{\mathbb{N}}
\newcommand{\mZ}{\mathbb{Z}}

\newcommand{\mC}{\mathbb{C}}
\newcommand{\mR}{\mathbb{R}}

\newcommand{\End}{\mathrm{End}}

\newcommand{\Ob}{\mathrm{Ob}}
\newcommand{\Mor}{\mathrm{Mor}}
\newcommand{\Ext}{\mathrm{Ext}}

\newcommand{\Hom}{\mathrm{Hom}}

\newcommand{\Sym}{\mathrm{Sym}}
\newcommand{\id}{\mathrm{id}}

\newcommand{\St}{\mathrm{St}}

\newcommand{\perf}{\mathrm{perf}}

\newcommand{\Lie}{\mathrm{Lie}}

\newcommand{\sVec}{\mathsf{sVec}}

\newcommand{\Tilt}{\mathsf{Tilt}}

\newcommand{\height}{\operatorname{ht}}

\newcommand{\mF}{\mathbb{F}}

\newcommand{\Fr}{\mathrm{Fr}}

\newcommand{\GL}{\mathrm{GL}}
\newcommand{\PGL}{\mathrm{PGL}}

\newcommand{\SL}{\mathrm{SL}}

\newcommand{\sM}[1]{\left(\begin{smallmatrix}#1\end{smallmatrix}\right)}

\begin{document}
\title{Higher Verlinde categories of reductive groups}
\author{Joseph Newton}
\address{School of Mathematics and Statistics, University of Sydney, Australia}
\email{j.newton@maths.usyd.edu.au}
\date{\today}


\setcounter{tocdepth}{1}

\begin{abstract}
We define tensor categories $\Ver_{p^n}(G)$ in characteristic $p$ for connected reductive groups $G$ and positive integers $n$, generalising the semisimple Verlinde categories $\Ver_p(G)$ originating from Gelfand-Kazhdan and the higher Verlinde categories $\Ver_{p^n}$ for $\SL_2$ defined by Benson-Etingof-Ostrik. The construction is based on the definition of $\Ver_{p^n}$ as an abelian envelope of a quotient of a category of tilting modules, but we also introduce an expanded construction which refines the $\SL_2$ case and gives new results. In particular, the union $\Ver_{p^\infty}(G)$ can be derived from the perfection of $G$; certain exact sequences in $\Rep G$ map to exact sequences in $\Ver_{p^n}(G)$; and the underlying abelian category of $\Ver_{p^n}$ can be expressed as a subcategory of $\Rep\SL_2$, or as a Serre quotient of a subcategory of $\Rep\SL_2$. 
\end{abstract}

\maketitle

\tableofcontents

\section*{Introduction}

The semisimple Verlinde category $\Ver_p$, a symmetric tensor category over a field $\bk$ of characteristic $p>0$, can be defined as $\Tilt\SL_2/\cI$ where $\Tilt\SL_2$ is the category of tilting modules of $\SL_2$ and $\cI$ is the tensor ideal of negligible morphisms. There are two known generalisations of this category, which are each important to the study of symmetric tensor categories in different ways:
\begin{enumerate}
\item  Replacing $\SL_2$ with an arbitrary simple algebraic group $G$ produces a semisimple symmetric tensor category $\Ver_p(G)$. These categories were introduced in \cite{GK}, and their structure has been examined further in \cite{CEN}. In particular, $\Ver_p(G)$ fibres over $\Ver_p$ via restriction to a principal $\SL_2$, and can alternately be constructed by taking representations of the Lie algebra of $G$ in $\Ver_p$. It is conjectured in \cite{asymptotic} that all finitely-generated semisimple symmetric tensor categories of moderate growth can be constructed from $\Ver_p(G)$ for various $G$ via Deligne products, equivariantisation, and changes of braiding. Similar categories exist in the non-symmetric setting in characteristic zero via quantum groups, see \cite{AP} and \cite[\S 8.18.2]{EGNO}.
\item  Replacing $\cI$ with a different tensor ideal in $\Tilt\SL_2$ produces an additive monoidal category which has an abelian envelope denoted $\Ver_{p^n}$. These envelopes were shown to exist in \cite{MAE} and \cite{BEO}, and their structure is examined in the latter paper. In particular, $\Ver_{p^n}$ is incompressible, non-semisimple for $n\geq2$, and has an inclusion functor $\Ver_{p^n}\hookrightarrow\Ver_{p^{n+1}}$ for $n\geq1$ allowing the union $\Ver_{p^\infty}$ to be constructed. It is conjectured in \cite{BEO} that all symmetric tensor categories of moderate growth fibre over $\Ver_{p^\infty}$, which would effectively reduce the study of such categories to the study of affine group schemes in $\Ver_{p^\infty}$ by the theory of \cite{incompressible}. More general categories exist in the non-symmetric setting via the quantum group of $\SL_2$ in positive characteristic, see \cite{STWZ} and \cite{decoppet}.
\end{enumerate}

In this paper we define categories $\Ver_{p^n}(G)$ combining these two generalisations. This is made possible by a theorem of \cite{Stroinski} (reformalised in \cite{tensorideals2}) which guarantees the existence of an abelian envelope for categories with a minimal tensor ideal under certain conditions. Our initial definition and construction of $\Ver_{p^n}(G)$ is similar to that of $\Ver_{p^n}$ in \cite{BEO}, however we also show that we may replace $\Tilt G$ in the construction with a bigger category $\overline\cT_n\subset\Rep G$, which is large enough that for $G=\SL_2$ the functor $\overline\cT_n\to\Ver_{p^n}$ is essentially surjective. This allows us to prove many results of \cite{BEO} much more efficiently while extending them to arbitrary $G$, for instance showing that the inclusion functor $\Ver_{p^n}(G)\hookrightarrow\Ver_{p^{n+1}}(G)$ arises directly from the Frobenius twist in $\Rep G$. This method also gives new insights into exact sequences in $\Ver_{p^n}$, including those defining symmetric and exterior powers of certain objects (see Proposition~\ref{PropExactSeq}), which have previously been elusive.

The main theorem of this paper is as follows:

\begin{thm}\label{ThmMain}
Let $\bk$ be an algebraically closed field with characteristic $p>0$, and let $G$ be a connected reductive linear algebraic group such that the Coxeter number $h$ of $G$ satisfies $p\geq\max(h,2h-4)$. Fix a principal morphism $\phi:\SL_2\to\widetilde G$ to a simply connected cover $\widetilde G\twoheadrightarrow G$, giving a functor $F:\Tilt G\to\Tilt\SL_2$ (see Section~\ref{SecPrincipal}). For an integer $n\geq1$, if $\cI_n(\SL_2)$ is the tensor ideal in $\Tilt\SL_2$ such that the abelian envelope of $\Tilt\SL_2/\cI_n(\SL_2)$ is $\Ver_{p^n}$, then $\Tilt G/F^{-1}(\cI_n(\SL_2))$ has an abelian envelope $\Ver_{p^n}(G)$ with the following properties:
\begin{enumerate}
\item\label{ItemProj} The indecomposable projective objects of $\Ver_{p^n}(G)$ are the images of tilting modules $T(\lambda)$ with $\lambda\in((p^{n-1}-1)\rho+\Lambda_{n-1}+p^{n-1}A)\cap X(T)$, where $A$ is the fundamental alcove and $\Lambda_{n-1}$ is the set of $p^{n-1}$-restricted weights as defined in Section~\ref{SecGCover}. In particular, if $G$ is simple then $\Ver_{p^n}(G)$ is a finite tensor category.
\item\label{ItemAbEnv} $\Ver_{p^n}(G)$ is also an abelian envelope of $\overline\cT_n/\overline\cI_n$, where $\overline\cT_n$ is the full subcategory of $\Rep G$ consisting of objects $X$ such that $X\otimes\St_{n-1}$ is a tilting module (of a cover of $G$ in which the Steinberg module $\St_{n-1}$ is well-defined), and $\overline\cI_n$ is a tensor ideal in $\overline\cT_n$ defined in Section~\ref{SecBarDef}.
\item\label{ItemSimps} The simple objects in $\Ver_{p^n}(G)$ are the images of the simple modules $L(\lambda)\in\overline\cT_n$ for $\lambda\in(\Lambda_{n-1}+p^{n-1}A)\cap X(T)$. These satisfy a Steinberg tensor product theorem, meaning that if $\lambda=\lambda_0+p\lambda_1+\cdots+p^{n-1}\lambda_{n-1}$ with $\lambda_i\in\Lambda_1\cap X(T)$ for all $i$ and $\lambda_{n-1}\in A$, then $L(\lambda)=L(\lambda_0)\otimes L(p\lambda_1)\otimes\cdots\otimes L(p^{n-1}\lambda_{n-1})$ in $\Ver_{p^n}(G)$.
\item\label{ItemPrin} There are tensor functors $\Ver_{p^n}(G)\to\Ver_{p^n}$ commuting with the restriction functors $\overline\cT_n(G)\to\overline\cT_n(\SL_2)$ coming from $\phi$.
\item\label{ItemFrob} There are inclusion functors $\Ver_{p^n}(G)\hookrightarrow\Ver_{p^{n+1}}(G)$ commuting with the Frobenius twist functors $\overline\cT_n\to\overline\cT_{n+1}$, $X\mapsto X^{(1)}$.
\item\label{ItemDecomp} If $p>h$, then we have $\Ver_{p^n}(G)\simeq\Ver_{p^n}(G/Z)\boxtimes\Rep_{\sVec}(Z,z)$, where $Z$ is the centre of $G$ and $z$ is the map $\mZ/2\to Z$ found by restricting $\phi$ to the centres.
\item\label{ItemPerf} The union $\Ver_{p^\infty}(G)\coloneqq\bigcup_{n=1}^\infty\Ver_{p^n}(G)$ via the inclusions $\Ver_{p^n}(G)\to\Ver_{p^{n+1}}(G)$ is the abelian envelope of $\overline\cT_\infty/\overline\cI_\infty$ for a full subcategory $\overline\cT_\infty$ of $\Rep(G_{\rm perf})$ and tensor ideal $\overline\cI_\infty$ in $\overline\cT_\infty$ defined in Section~\ref{SecPerf}, where $G_{\rm perf}$ is the perfection of $G$.
\item\label{ItemExact} The functor $\overline\cT_n\to\Ver_{p^n}(G)$ sends bounded exact sequences to exact sequences.
\end{enumerate}
\end{thm}

The condition $p\geq 2h-4$ ensures that Donkin's tensor product theorem holds (as well as some related results, see \cite{dconj3}), while $p\geq h$ ensures that the fundamental alcove of $G$ is non-empty. Conjecturally, the former condition could be loosened without changing the theorem above. On the other hand, taking $p<h$ is expected to give meaningfully different results, and the papers \cite{BEEO} and \cite{BT} have developments in this direction for $n=1$.

The categories $\Ver_{p^n}(G)$ for $n\geq2$ are some of the first explicit examples of tensor categories that fibre over $\Ver_{p^n}$. An important task for future study would be to understand these categories from the perspective of Tannakian formalism as in \cite[\S 4]{incompressible}, or in other words to describe the affine group scheme in $\Ver_{p^n}$ whose category of representations is equivalent to $\Ver_{p^n}(G)$. This has been done for $n=1$ in \cite[\S 3.3]{CEN} using the methods of \cite{Ve}, however no similar methodology currently exists for $n\geq2$. Another avenue for development is to generalise the results of \cite{decoppet}, regarding higher Verlinde categories of the quantum group of $\SL_2$, to arbitrary reductive groups. We show that such a generalisation exists for $p\geq 2h-2$ in Remark~\ref{RemQuantum}, however its properties deserve further study.

For $G=\SL_2$, the methods introduced in this paper also give a more explicit description of the underlying abelian category of $\Ver_{p^n}$:

\begin{thm}\label{ThmSL2}
We define the following full subcategories of $\Rep\SL_2$:
\begin{enumerate}
\item $\cA_n$ consists of objects with weights strictly less than $p^n-1$;
\item $\cB_n$ is the Serre subcategory of $\cA_n$ with simple objects $L_i$ for $(p-1)p^{n-1}\leq i<p^n-1$;
\item $\cC_n$ consists of objects $X\in\cA_n$ with $\Hom(X,B)=0=\Hom(B,X)$ for all $B\in\cB_n$.
\end{enumerate}
Then we have equivalences of abelian categories
\[\cC_n\simeq\cA_n/\cB_n\simeq\Ver_{p^n}\]
where $\cA_n/\cB_n$ denotes the Serre quotient.
\end{thm}

The paper is organised as follows. In Section~\ref{SecPrelim} we give some background on algebraic groups and tensor categories, and explain the assumptions on $G$ in Theorem~\ref{ThmMain}. In Section~\ref{SecDefinition} we define $\Ver_{p^n}(G)$ using a subcategory $\cT_n$ of $\Tilt G$, and show that it fibres over $\Ver_{p^n}$. In Section~\ref{SecExpansion} we expand this construction to a larger subcategory $\overline\cT_n$ of $\Rep G$, and show that this matches the definition in Theorem~\ref{ThmMain}. In Section~\ref{SecProperties} we use this expanded construction to derive the properties of $\Ver_{p^n}(G)$ listed in Theorem~\ref{ThmMain}. In Section~\ref{SecVerpnSL2} we consider $G=\SL_2$ and prove Theorem~\ref{ThmSL2}.

\subsection*{Acknowledgements} This work is part of PhD research supervised by Kevin Coulembier, whom the author thanks for his guidance and feedback. The author also thanks Geordie Williamson for further supervision and discussions regarding $\SL_3$ and $U_q(\mathfrak{sl}_3)$, Thibault D\'ecoppet for discussions regarding quantum groups and mixed Verlinde categories, and a reviewer for their helpful comments. This version includes corrections after proofreading with LLM programs; however, no major structural changes or substantive mathematical ideas were introduced in the proofreading process, and all edits were made directly by the author. This work was supported by an Australian Government Research Training Program Scholarship, and hindered by the release of \textit{Hollow Knight: Silksong}.

\section{Preliminaries and assumptions}\label{SecPrelim}

\subsection{}\label{SecAlgGrps} Throughout this paper, $\bk$ is an algebraically closed field of characteristic $p>0$, and $G$ is an affine group scheme over $\bk$ with the following properties:
\begin{enumerate}
\item $G$ is a connected reductive linear algebraic group;
\item $p\geq h$ where $h$ is the Coxeter number of $G$; and
\item Donkin's tensor product theorem holds.
\end{enumerate}
To clarify, the first condition means that $G$ is of finite type, reduced, irreducible, and has no non-trivial normal smooth connected unipotent subgroups. The latter two conditions are elaborated on below.

\subsection{The cover $\widetilde G$ and its weights}\label{SecGCover} $G$ has a maximal torus $T$ and weight space $X(T)$. Since we may require a larger weight space, we will define a cover $\widetilde G$ with a ``maximal'' weight space. By \cite[\S II.1.18]{Jantzen} there are tori $T_1,T_2\subseteq T$ such that we have a quotient map $\cD G\times T_2\twoheadrightarrow G$ with finite kernel $T_1\cap T_2$, where $\cD G$ is the derived subgroup of $G$. Moreover, $\cD G$ has a simply-connected cover $\widetilde{\cD G}$. We then define $\widetilde G\coloneqq\widetilde{\cD G}\times T_2$ and write $\Lambda$ for the weight space of $\widetilde G$, so that $\widetilde G\twoheadrightarrow G$ is a cover in the sense of \cite[\S II.1.17]{Jantzen} and we have a canonical inclusion $X(T)\subseteq\Lambda$.

The groups $G$ and $\widetilde G$ have the same root system $\Phi$, Weyl group $W$ and affine Weyl group $\widetilde W_p$. We fix a Borel subgroup $B\subseteq G$ containing $T$, giving a set of positive roots $\Phi^+\subseteq\Phi$ and simple roots $\alpha_1,\dots,\alpha_r$ where $r$ is the rank of $\cD G$ ($\Phi$ is empty if $G$ is a torus). This gives sets of dominant weights $X(T)^+$ and $\Lambda^+$ for $G$ and $\widetilde G$ respectively. For $n\in\mN$ we write
\[\Lambda_n=\{\lambda\in\Lambda^+\mid\langle\lambda,\alpha_i^\vee\rangle<p^n\text{ for }1\leq i\leq r\}\] 
for the set of $p^n$-restricted weights, and we write $X_n(T)=\Lambda_n\cap X(T)$.

We write $\varpi_1,\dots,\varpi_r\in\Lambda$ for the fundamental weights of $\widetilde{\cD G}$ (which are elements of $\Lambda$ via the quotient of $\widetilde G$ by $T_2$), and the Weyl vector is $\rho=\varpi_1+\cdots+\varpi_r=\frac12\sum_{\alpha\in\Phi^+}\alpha$. The $\rho$-shifted action of $\widetilde W_p$ on $\Lambda$ divides $\Lambda\otimes_\mZ\mR$ into alcoves, and we write $A\subseteq\Lambda^+$ for the weights in the fundamental alcove (meaning the alcove containing 0). We also write $\overline A\subseteq\Lambda^+$ for the weights in the upper closure of the fundamental alcove. If $G$ is a torus then we have $\rho=0$ and $A=\overline A=\Lambda=\Lambda^+$, and we set the Coxeter number $h$ to be 0. Otherwise, the root system $\Phi$ is a union of irreducible root systems $\Phi_1\cup\cdots\cup\Phi_m$, and by \cite[\S 3.1.2]{CEN} we have
\begin{align*}
A&=\{\lambda\in\Lambda^+\mid\langle\lambda+\rho,\theta_j^\vee\rangle<p\text{ for }1\leq j\leq m\}\\
\overline A&=\{\lambda\in\Lambda^+\mid\langle\lambda+\rho,\theta_j^\vee\rangle\leq p\text{ for }1\leq j\leq m\}
\end{align*}
where $\theta_j$ is the highest short root of $\Phi_j$. The Coxeter number of $\Phi_j$ is
\[h_j=1+\langle\rho,\theta_j^\vee\rangle=1+\sum_{i=1}^r k_i,\quad\text{where }\theta_j^\vee=\sum_{i=1}^r k_i\alpha_i^\vee\]
by \cite[\S 3.1.2]{CEN} again, and we define the Coxeter number of $G$ to be $h=\max\{h_1,\dots,h_m\}$. The condition $p\geq h$ ensures that $0\in A$, so $A$ is non-empty.

\subsection{Modules of $G$} We write $L(\lambda)$, $\Delta(\lambda)$, and $T(\lambda)$ for the simple, Weyl, and indecomposable tilting modules respectively of $\widetilde G$ with highest weight $\lambda\in\Lambda^+$. If $\lambda\in X(T)^+$, then these are also simple, Weyl and indecomposable tilting modules respectively of $G$. Recall that $L(\lambda)^*=L(-w_0\lambda)$ and $T(\lambda)^*=T(-w_0\lambda)$ where $w_0$ is the longest element of $W$ as a Coxeter group (see \cite[E.6]{Jantzen}). We write $\Rep G$ for the category of finite-dimensional representations of $G$, and $\Tilt G$ for the full subcategory of tilting modules. By the linkage principle, we have $T(\lambda)=L(\lambda)$ whenever $\lambda\in \overline A$.  For $n\in\mN$, the $n$-th Steinberg module is $\St_n(G)=L((p^n-1)\rho)=T((p^n-1)\rho)$. We just write $\St_n$ if $G$ is clear from context. A priori, $\St_n$ is only a module of $\widetilde G$, but if $p$ is odd then we have $(p^n-1)\rho\in X(T)$ automatically (since $\Phi\subset X(T)$) and $\St_n$ is a module of $G$. Examples of $G$ where $\St_n$ is not defined over $G$ in characteristic 2 are $\GL_2$ and $\PGL_2$.

For $n\in\mN$ and a representation $X$ of $G$ or $\widetilde G$, we write $X^{(n)}$ for the $n$-th Frobenius twist of $X$. Recall Steinberg's tensor product theorem \cite[Corollary II.3.17]{Jantzen}:
\[L(\lambda+p^n\mu)\cong L(\lambda)\otimes L(\mu)^{(n)}\text{ for $\lambda\in\Lambda_n$ and $\mu\in\Lambda^+$.}\]
By \textbf{Donkin's tensor product theorem} we mean the identity
\[T(\lambda+p^n\mu)\cong T(\lambda)\otimes T(\mu)^{(n)}\text{ for $\lambda\in(p^n-1)\rho+\Lambda_n$ and $\mu\in\Lambda^+$.}\]
This is known to hold for $p\geq2h-4$ by \cite{dconj3}, since it follows from Donkin's Tilting Module Conjecture by \cite[Lemma~E.9]{Jantzen}.

\subsection{Principal $\SL_2$ morphisms}\label{SecPrincipal} The unipotent variety of $G$ and nilpotent cone of $\Lie(G)$ each have a unique dense orbit, and elements of these orbits are called \textbf{regular}. Regular unipotent elements $u\in G$ are equivalently defined by their centraliser $C_G(u)$ having dimension equal to the rank of $G$, which is the dimension of $T$ \cite[\S 1.14]{Carter} (note that for $p=h$ the same is not necessarily true for regular nilpotent elements of $\Lie(G)$). We call a morphism $\phi:\SL_2\to G$ \textbf{principal} if the differential $d\phi:\mathfrak{sl}_2\to\Lie(G)$ sends $\sM{0&1\\0&0}$ to a regular nilpotent element $e$ and $t\mapsto\phi\sM{t&0\\0&t^{-1}}$ is a cocharacter associated to $e$, in the sense of \cite[\S 2.4]{McNinch} or \cite[\S 4.2]{A1sub}. If $G$ is a torus (that is $G=T_2$) then the identity element is regular, and thus any trivial morphism $\SL_2\to G$ is principal. Otherwise, take $u$ to be a regular unipotent element in the semisimple subgroup $\cD G$. We recall from \cite[\S 0.4]{Testerman} that the order of $u$ is $\min\{p^n\mid n\in\mN,p^n>\height(P)\}$, where the value $\height(P)$ is determined by a particular parabolic subgroup $P$ containing $u$. This value is at most $h-1$ by the formula in Section~\ref{SecGCover}, so our assumption $p\geq h$ from Section~\ref{SecAlgGrps} implies that the order of $u$ is $p$. We can then apply \cite[Theorem~2]{McNinch} to obtain a principal morphism $\SL_2\to G$ with $u$ in its image.

One can construct a regular element $e$ of $\Lie(G)$ as follows. Let $e_i$ be root vectors corresponding to the simple roots $\alpha_i$. As described in \cite[\S 5]{Springer}, $\sum_i e_i$ is a regular element, and $2\rho^\vee=\sum_{\alpha\in\Phi^+}\alpha^\vee$ is an associated cocharacter by \cite[Definition~4.15]{A1sub}, where $\rho^\vee$ is the Weyl vector of the dual root system $\Phi^\vee$. Since regular nilpotent elements form a single orbit, all principal morphisms must induce a conjugate of this cocharacter. Consequently, if a principal morphism $\phi$ restricts to a map $\phi':\mG_m\to T$ on the diagonal matrices of $\SL_2$, then for a suitable choice of Borel $B$ and positive roots $\Phi^+$, $\phi$ induces a map $\phi^*:X(T)\to\mZ$ (identifying the weight space of $\SL_2$ with $\mZ$) given by
\[\phi^*(\lambda)=\langle\lambda,\phi'\rangle=\sum_{\alpha\in\Phi^+}\langle\lambda,\alpha^\vee\rangle.\]
This proves the weight formula in \cite[Proposition~3.2.1(3)]{CEN} for all $G$ satisfying the assumptions of Section~\ref{SecAlgGrps}. Additionally, we have
\[\phi^*(\alpha)=\sum_{\beta\in\Phi^+}\langle\alpha,\beta^\vee\rangle=2\langle\alpha,\rho^\vee\rangle\leq 2h-2\]
for each root $\alpha\in\Phi$. Since $p\geq h$, this means that principal morphisms are ``good'' in the sense of \cite[\S 5.1]{A1sub}, and thus by \cite[Lemma~5.3, Theorem~5.4(ii)]{A1sub} and conjugacy of regular elements, all principal morphisms $\SL_2\to G$ are conjugate to each other in $G$.

Now suppose $G$ has an $\mF_p$-rational structure, and equip $\SL_2$ with the standard $\mF_p$-rational structure. For a regular unipotent $u\in G(\mF_p)$, \cite[Theorem~23]{McNinch} constructs a morphism $\SL_2\to G$ compatible with the restriction to $\mF_p$ and with non-zero differential. This morphism is ``nice'' in the sense of \cite[\S 5.2]{McNinch}, so regularity of $u$ implies regularity of $d\phi\sM{0&1\\0&0}$ and so $\phi$ is principal. Thus, for any $p\geq h$ and any $\mF_p$-rational structure on $G$, there exists a principal morphism $\SL_2\to G$ which commutes with the Frobenius morphisms on $G$ and $\SL_2$. Conversely, since all principal morphisms are conjugate, every principal morphism is compatible with some $\mF_p$-rational structure on $G$.

\subsection{Tensor categories and ideals}\label{SecTensorCats} The definitions below follow \cite{tensorideals2}. A \textbf{pseudo-tensor category} is an essentially small $\bk$-linear additive Karoubi (i.e. idempotent-complete) category $\cT$ with a monoidal structure $\otimes$ such that $\otimes$ is rigid, $\End(\unit)\cong\bk$, and $\Hom(X,Y)$ is finite-dimensional for all $X,Y\in\cT$. We write $\id_X$, $\ev_X$ and $\coev_X$ for the identity, evaluation and coevaluation morphisms on $X\in\cT$. A \textbf{pseudo-tensor functor} is a faithful $\bk$-linear monoidal functor between pseudo-tensor categories. A \textbf{tensor category} is an abelian pseudo-tensor category such that all objects have finite length, and a \textbf{tensor functor} is an exact $\bk$-linear monoidal functor between tensor categories (which is necessarily faithful by \cite[2.10]{Del01}). In particular, $\Rep G$ is a tensor category and $\Tilt G$ is a pseudo-tensor category, and both of these also have a symmetric braiding. An \textbf{inclusion} of tensor categories is a full tensor functor that sends simples to simples (such functors are called ``injective'' in \cite{incompressible}). Equivalently, an inclusion $\cC\hookrightarrow\cC'$ is an equivalence of $\cC$ with a tensor subcategory of $\cC'$, meaning a full subcategory closed under direct sums, subquotients, tensor products and duals (see \cite[Lemma~3.1.1]{incompressible}). For example, the functor $\Rep G\to\Rep\widetilde G$ coming from $\widetilde G\twoheadrightarrow G$ is an inclusion functor.

A \textbf{tensor ideal} $\cI$ in a pseudo-tensor category $\cT$ is a collection of $\bk$-vector spaces $\cI(X,Y)\subseteq\Hom(X,Y)$ for $X,Y\in\cT$ that is closed under composing or tensoring with any morphism in $\cT$. A \textbf{thick tensor ideal} $I$ in a pseudo-tensor category $\cT$ is a collection of objects in $\cT$ that is closed under tensoring with any object in $\cT$, and also closed under isomorphisms, direct sums and direct summands. The preimage of any (thick) tensor ideal via a pseudo-tensor functor is also a (thick) tensor ideal, and any tensor ideal $\cI$ in $\cT$ has an associated thick tensor ideal $\Ob(\cI)$ consisting of objects $X\in\cT$ whose identity morphisms $\id_X$ are in $\cI$. Given a thick tensor ideal $I$ in $\cT$ there are unique minimal and maximal tensor ideals $I^{\rm min}$ and $I^{\rm max}$ among tensor ideals $\cI$ with $\Ob(\cI)=I$. Specifically, $I^{\rm min}$ consists of all morphisms in $\cT$ that factor through an object in $I$, and $I^{\rm max}$ is the sum of all tensor ideals $\cI$ with $\Ob(\cI)=I$ (see \cite[\S 2.3.1]{tensorideals2}).

For a proper tensor ideal $\cI$ in a pseudo-tensor category $\cT$, we define the category $\cT/\cI$ to have the same objects as $\cT$ but with morphism spaces $\Hom_{\cT}(X,Y)/\cI(X,Y)$. Then $\cT/\cI$ is also a pseudo-tensor category. Note that $X$ is isomorphic to the zero object in $\cT/\cI$ if and only if $X\in\Ob(\cI)$. If $I=\Ob(\cI)$ and $J$ is another thick tensor ideal in $\cT$ containing $I$, then the image of $J$ in $\cT/\cI$ is also a thick tensor ideal, and we write $J/I$ for this image. The non-zero indecomposable objects in $J/I$ are indecomposables $X\in J\setminus I$.

\subsection{Tensor ideals in $\Tilt\SL_2$} \label{SecSL2Tensor} If $G=\SL_2$ then we can identify $X(T)$ with $\mZ$, and we write $L_i$, $\Delta_i$, $T_i$ for the simple, Weyl and tilting modules corresponding to $i\in\mN$. Then Steinberg's tensor product theorem becomes
\[L_{a+bp^r}=L_a\otimes L_b^{(r)}\text{ for $0\leq a\leq p^r-1$ and $b\in\mN$}\]
while Donkin's tensor product theorem (which holds for all $p>0$) becomes
\[T_{a+bp^r}=T_a\otimes T_b^{(r)}\text{ for $p^r-1\leq a\leq 2p^r-2$ and $b\in\mN$.}\]
We also have $T_p=T_1\otimes T_{p-1}$, and for $p\leq a\leq 2p-2$ the module $T_a$ is uniserial with Weyl composition factors $[\Delta_{2p-2-a},\Delta_a]$, or simple composition factors $[L_{2p-2-a},L_a,L_{2p-2-a}]$ (see \cite[\S 3]{BEO} and \cite[\S E.1]{Jantzen}). This allows us to inductively deduce the weights and simple composition factors of every indecomposable tilting module of $\SL_2$.

We recall from \cite[Theorem~5.3.1]{tensorideals1} or \cite[Proposition~3.5]{BEO} that we have a tensor ideal $\cI_n(\SL_2)$ in $\Tilt\SL_2$ consisting of morphisms that are supported on summands with highest weights greater than or equal to $p^n-1$, and these comprise all tensor ideals in $\Tilt\SL_2$. We also recall from \cite[Proposition~4.4]{GK} that if $F:\Rep G\to\Rep\SL_2$ is restriction along a principal morphism $\phi:\SL_2\to G$, then $F(T(\lambda))\in\Ob(\cI_1(\SL_2))$ if and only if $\lambda\not\in A$. This is stated only for simply-connected simple groups, however the proof applies to arbitrary $G$ by the reasoning in Section~\ref{SecPrincipal}. In particular, since $\sum_{\alpha\in\Phi^+}\alpha^\vee=2\rho^\vee$ where $\rho^\vee$ is the Weyl vector of the dual root system $\Phi^\vee$, the restriction of any simple root $\alpha_i$ of $G$ to a weight of $\SL_2$ along $\phi$ is $\phi^*(\alpha_i)=2$, which matches the assumptions in \cite[\S 3.5]{GK}.

\subsection{Abelian envelopes}\label{SecBoxTensor} An \textbf{abelian envelope} of a pseudo-tensor category $\cT$ over $\bk$ is a tensor category $\cC$ over $\bk$ along with a pseudo-tensor functor $F:\cT\to\cC$ such that, for any tensor category $\cC'$ over $\bk$, composition with $F$ is an equivalence between the category of tensor functors $\cC\to\cC'$ and the category of pseudo-tensor functors $\cT\to\cC'$. Examples of abelian envelopes in the literature include \cite{MAE,EHS,BEO,quoprop}. Note that \cite{BEO} only considers envelopes where $F$ is full, however we will not require this (see Remark~\ref{RemNonFull}).

Given two pseudo-tensor categories $\cT_1$ and $\cT_2$, we define $\cT_1\otimes\cT_2$ and $\cT_1\dot\otimes\cT_2$ following \cite[\S 6]{quoprop}. The category $\cT_1\otimes\cT_2$ has objects $(X,Y)$ for $X\in\cT_1$, $Y\in\cT_2$ and morphisms
\[\Hom_{\cT_1\otimes\cT_2}((X,Y),(X',Y'))=\Hom_{\cT_1}(X,X')\otimes_\bk\Hom_{\cT_2}(Y,Y')\]
while the category $\cT_1\dot\otimes\cT_2$ is the additive Karoubi envelope of $\cT_1\otimes\cT_2$. If $G_1$ and $G_2$ are reductive algebraic groups, then by \cite[\S E.7]{Jantzen} the pseudo-tensor categories $\Tilt(G_1\times G_2)$ and $\Tilt(G_1)\dot\otimes\Tilt(G_2)$ are equivalent via $T(\lambda)\mapsto (T(\lambda_1), T(\lambda_2))$, where $\lambda_i$ is the projection of $\lambda$ to the weight space of $G_i$.

If $\cC_1$ and $\cC_2$ are tensor categories, then by \cite[Theorem~6.1.3]{quoprop} we can define the Deligne tensor product $\cC_1\boxtimes\cC_2$ as an abelian envelope of $\cC_1\dot\otimes\cC_2$, if it exists. If $C$ and $D$ are $\bk$-coalgebras, then we have $\Comod C\boxtimes\Comod D\simeq\Comod(C\otimes_\bk D)$ by \cite[Proposition~1.11.2]{EGNO}.

\subsection{Existence and construction of abelian envelopes}\label{SecAbEnv} The following definitions and facts are from \cite[\S 2]{BEO}. A morphism $f:X\to Y$ in an additive category is called \textbf{split} if it has an epi-mono factorisation $f=\iota\pi$ with both $\pi$ and $\iota$ split, meaning $\pi$ is a projection onto a direct summand of $X$ and $\iota$ is an inclusion of a direct summand of $Y$. A \textbf{splitting object} in a pseudo-tensor category $\cT$ is an object $P$ such that $f\otimes\id_P$ is split for every morphism $f$ in $\cT$. Splitting objects form a thick tensor ideal, and if $\cT$ is a tensor category with enough projectives then the splitting objects in $\cT$ are exactly the projective objects. We say that $\cT$ is \textbf{separated} if for any non-zero morphism $f$ in $\cT$ we have $f\otimes\id_P\neq0$ for some splitting object $P$.

Let $S$ be a thick tensor ideal in a pseudo-tensor category $\cT$, and let $\{P_i,i\in I\}$ be a complete set of representatives of isomorphism classes of indecomposable objects in $S$. Suppose that for each $i\in I$ there are finitely many $j\in I$ with $\Hom(P_i,P_j)\neq0$. As described in \cite[\S 2]{BEO}, we may construct a coalgebra $C=\bigoplus_{i,j\in I}\Hom(P_i,P_j)^*$ such that the algebra structure on the dual $C^*=\prod_{i,j\in I}\Hom(P_i,P_j)$ is $(f,g)\mapsto g\circ f$. We then have an abelian category $\cC(\cT,S)=\Comod(C)$ and a $\bk$-linear functor $F:\cT\to\cC(\cT,S)$ given by $X\mapsto\bigoplus_{i\in I}\Hom(P_i,X)$ where $C^*$ acts on $F(X)$ by composition on the right. If $S$ is the ideal of splitting objects, then it is shown in \cite[\S 2]{BEO} that $\cC(\cT,S)$ inherits a rigid monoidal structure from $\cT$.

\begin{theorem}\label{ThmAbEnv}
Let $S$ be the ideal of splitting objects in a pseudo-tensor category $\cT$. Then $\cT$ is separated if and only if the functor $F$ defined above is faithful. If this holds, then $\cC(\cT,S)$ is a tensor category and $F:\cT\to\cC(\cT,S)$ is an abelian envelope of $\cT$.
\end{theorem}

\begin{proof}
Separatedness is characterised by \cite[Proposition~2.25]{BEO}, and then it is shown in \cite[Proposition~2.37]{BEO} that $\cC(\cT,S)$ is a tensor category and hence $F$ is a pseudo-tensor functor. The case where $F$ is full is proven in \cite[Theorem~2.42]{BEO}, and for completeness we will reprove this for $F$ not necessarily full. We will show $\cC(\cT,S)$ satisfies \cite[Corollary~4.4.4]{envelopes3}. Condition (G) in this corollary is satisfied, since all objects in $\Comod(C)$ are quotients of sums of $F(P_i)$ for $i\in I$. For condition (F), notice that the map $\Hom(P_j,X)\to\Hom(F(P_j),F(X))$ is surjective for any $j\in I$ and $X\in\cT$, since if
\[f:F(P_j)=\bigoplus_{i\in I}\Hom(P_i,P_j)\to F(X)=\bigoplus_{i\in I}\Hom(P_i,X),\]
is a $C^*$-module homomorphism then $f=f(\id_{P_j})\circ-$. For a morphism $a:F(X)\to F(Y)$ in $\cC(\cT,S)$, choose an epimorphism $F(P)\twoheadrightarrow F(X)$ where $P$ is a direct sum of objects in $\{P_i\mid i\in I\}$. The surjectivity above means this has a preimage $q:P\to X$, and $aF(q)$ also has a preimage as required.
\end{proof}

The following theorem is \cite[Theorem~2.4.1(2)]{tensorideals2}, which guarantees existence of an abelian envelope for certain quotient categories.

\begin{theorem}\label{ThmCEOEnv}
Suppose $\cT$ is a pseudo-tensor category with an isomorphism $X\otimes Y\xrightarrow{\sim}Y\otimes X$ for $X,Y\in\cT$ that is natural in $X$ and $Y$, and $I$ is a thick tensor ideal such that
\begin{enumerate}
\item there is a unique thick tensor ideal $J$ strictly containing $I$ such that every thick tensor ideal of $\cT$ strictly containing $I$ also contains $J$, and
\item the number of isomorphism classes of indecomposable objects in $J/I$ is finite.
\end{enumerate}
Then $\cT/I^{\rm max}$ is separated with splitting objects $J/I$, and the tensor category $\cC(\cT/I^{\rm max},J/I)$ is an abelian envelope of $\cT/I^{\rm max}$.
\end{theorem}

Note that the functor $\cT\to\cC(\cT/I^{\rm max},J/I)$ is not necessarily faithful on objects whose indecomposable summands are all in $J\setminus I$, see for example Remark~\ref{RemNonFaithful}.

\section{Definition of $\Ver_{p^n}(G)$}\label{SecDefinition}

\subsection{}\label{SecTIJDef} Fix an integer $n\geq1$. We make the following definitions:
\begin{enumerate}
\item Let $\cT_n(G)$ be the full subcategory of $\Tilt G$ consisting of objects whose indecomposable summands have highest weights in $\{0\}\cup((p^{n-1}-1)\rho+\Lambda^+)$.
\item Let $I_n(G)$ be the class of objects in $\Tilt G$ whose indecomposable summands have highest weights in $(p^{n-1}-1)\rho+\Lambda_{n-1}+p^{n-1}(\Lambda^+\setminus A)$.
\item Let $J_n(G)$ be the class of objects in $\Tilt G$ whose indecomposable summands have highest weights in $(p^{n-1}-1)\rho+\Lambda^+$.
\end{enumerate}
If $G$ is clear from context, we will just write $\cT_n$, $I_n$ and $J_n$. Note that $\cT_n(G),I_n(G),J_n(G)$ are the restrictions/preimages of the corresponding categories/classes for $\widetilde G$ or $\widetilde{\cD G}$. The example $G=\SL_3$ is shown in Figure~\ref{FigSL3}. Also note that any indecomposable module in $J_n$ is isomorphic to $T(\lambda+p^{n-1}\mu)$ for some weights $\lambda\in(p^{n-1}-1)\rho+\Lambda_{n-1}$ and $\mu\in\Lambda^+$, and we have $T(\lambda+p^{n-1}\mu)\in I_n$ if and only if $\mu\not\in A$. By Donkin's tensor product theorem, this tilting module is isomorphic to $T(\lambda)\otimes T(\mu)^{(n-1)}$ as a $\widetilde G$-module. Under the assumptions in Section~\ref{SecAlgGrps}, $((p^{n-1}-1)\rho+\Lambda_{n-1})\cap X(T)$ contains the weight $2(p^{n-1}-1)\rho$, so there is always at least one indecomposable object in $J_n\setminus I_n$.

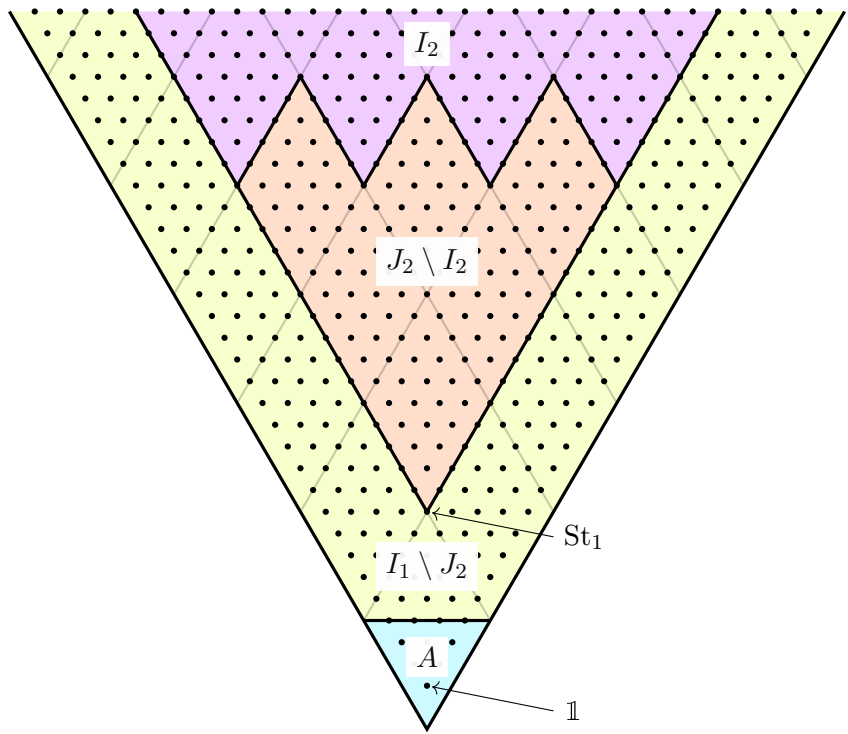
\begin{figure}[h]
\begin{tikzpicture}[scale=0.333]
\definecolor{color1}{HTML}{CCFAFF}
\definecolor{color2}{HTML}{F7FFCC}
\definecolor{color3}{HTML}{FFDFCC}
\definecolor{color4}{HTML}{F1CCFF}
\fill[color1] (0,-2*0.866) -- (-2.5,3*0.866) -- (2.5,3*0.866);
\fill[color2] (0,8*0.866) -- (-11.5,31*0.866) -- (-16.5,31*0.866) -- (-2.5,3*0.866) -- (2.5,3*0.866) -- (16.5,31*0.866) -- (11.5,31*0.866) -- (0,8*0.866);
\fill[color3] (0,8*0.866) -- (-7.5,23*0.866) -- (-5,28*0.866) -- (-2.5,23*0.866) -- (0,28*0.866) -- (2.5,23*0.866) -- (5,28*0.866) -- (7.5,23*0.866);
\fill[color4] (-11.5,31*0.866) -- (-7.5,23*0.866) -- (-5,28*0.866) -- (-2.5,23*0.866) -- (0,28*0.866) -- (2.5,23*0.866) -- (5,28*0.866) -- (7.5,23*0.866) -- (11.5,31*0.866);
\foreach \x in {1,...,6} { \draw[thick, opacity=0.2] (-2.5*\x, 4.330*\x-1.732) -- (16.5-5*\x, 31*0.866); }
\foreach \x in {1,...,6} { \draw[thick, opacity=0.2] (2.5*\x, 4.330*\x-1.732) -- (-16.5+5*\x, 31*0.866); }
\foreach \y in {0,...,31} { \foreach \x in {0,...,\y} { \fill (\x-\y/2, 0.866*\y) circle (3.5pt); }}
\draw[very thick] (-16.5,31*0.866) -- (0,-2*0.866) -- (16.5,31*0.866);
\draw[very thick] (-2.5,3*0.866) -- (2.5,3*0.866);
\draw[very thick] (-11.5,31*0.866) -- (0,8*0.866) -- (11.5,31*0.866);
\draw[very thick] (-7.5,23*0.866) -- (-5,28*0.866) -- (-2.5,23*0.866) -- (0,28*0.866) -- (2.5,23*0.866) -- (5,28*0.866) -- (7.5,23*0.866);
\node[fill=white, opacity=0.95] at (0,1.333*0.866) {$A$};
\node[fill=white, opacity=0.95] at (0,5.5*0.866) {$I_1\setminus J_2$};
\node[fill=white, opacity=0.95] at (0,19.5*0.866) {$J_2\setminus I_2$};
\node[fill=white, opacity=0.95] at (0,29.5*0.866) {$I_2$};
\draw[<-] (0.2,-0.04) -- (5,-1) node[right] {$\unit$};
\draw[<-] (0.2,8*0.866-0.04) -- (5,8*0.866-1) node[right] {$\St_1$};
\end{tikzpicture}
\caption{Dominant weights $X(T)^+$ for $\SL_3$ with $p=5$, their corresponding indecomposable tilting modules, and the thick tensor ideals those tilting modules belong to. Each labelled region includes all dominant weights on its lower boundary. The faint lines show translates of $X_1(T)$ shifted by $-\rho$.}\label{FigSL3}
\end{figure}

\begin{lemma}\label{LemThickIdeal}
$\cT_n$ is a monoidal subcategory of $\Tilt G$, and both $J_n$ and $I_n$ are thick tensor ideals in both $\Tilt G$ and $\cT_n$. Moreover, if $X\in J_n\setminus I_n$ then any $Y\in J_n$ is a summand of $X\otimes Z$ for some $Z\in J_n$. This means $I_n$ contains every thick ideal in $\Tilt G$ strictly contained in $J_n$, and $J_n$ and $\Ob(\cT_n)$ are the only thick tensor ideals in $\cT_n$ strictly containing $I_n$.
\end{lemma}

\begin{proof}
$I_n,J_n$ and $\Ob(\cT_n)$ are closed under taking duals. \cite[Lemma~E.8]{Jantzen} states that $T(\lambda)$ is projective over $G_{n-1}T$ if and only if $T(\lambda)\in J_n$, and the tensor product of a projective module with any module is projective by \cite[Proposition~4.2.12]{EGNO}. Thus $J_n$ is a thick tensor ideal, and consequently $\cT_n$ is monoidal. Indecomposables in $J_n$ are of the form $T(\lambda+p^{n-1}\mu)$ with $\lambda\in(p^{n-1}-1)\rho+\Lambda_{n-1}$ and $\mu\in\Lambda^+$, and we have $X\in I_n$ if and only if $\mu\in\Lambda^+\setminus A$. Working over the cover $\widetilde G$ if necessary, we have $T(\lambda)\otimes T(\lambda')\in J_n$ for any $\lambda'\in X(T)^+$, and so we can apply Donkin's theorem to each summand tensored with $T(\mu)^{(n-1)}$ to show $T(\lambda)\otimes T(\lambda')\otimes T(\mu)^{(n-1)}\in I_n$ whenever $\mu\in\Lambda^+\setminus A$, so $I_n$ is a thick tensor ideal.

For the second statement in the lemma, we may assume without loss of generality that $X$ and $Y$ are indecomposable, so let $X=T(\lambda+p^{n-1}\mu)\in J_n\setminus I_n$ and work over $\widetilde G$ for now. The fundamental alcove $A$ is a weight cell (see e.g. \cite[Proposition~12]{AnCells}), meaning there exists $\mu'\in A$ such that $\unit$ is a summand of $L(\mu)\otimes L(\mu')$. Let $\lambda'=-w_0(\lambda-(p^{n-1}-1)\rho)$, so that $T(\lambda)$ is a summand of $T(\lambda')^*\otimes\St_{n-1}$. The projection $T(\lambda')^*\otimes\St_{n-1}\twoheadrightarrow T(\lambda)$ induces a map $f:\St_{n-1}\to T(\lambda')\otimes T(\lambda)$ by adjunction, which is injective since $\St_{n-1}$ is simple. Suppose for a contradiction that $f$ is non-split. All tilting modules $T(\sigma)$ in the same block as $\St_{n-1}$ have $\sigma=(p^{n-1}-1)\rho+p^{n-1}\sigma'$ for some $\sigma'\in X(T)^+$, and if $\sigma'\in A\setminus\{0\}$ then $T(\sigma)\cong\St_{n-1}\otimes L(\sigma')^{(n-1)}$ is simple and not isomorphic to $\St_{n-1}$. Thus, $\St_{n-1}$ has no non-zero morphisms to any indecomposable object in $J_n\setminus I_n$ other than itself, and thus $f$ factors through a summand $S\subseteq T(\lambda')\otimes T(\lambda)$ such that $S\in I_n$. Now the composition
\[T(\lambda')^*\otimes\St_{n-1}\hookrightarrow T(\lambda')^*\otimes S\hookrightarrow T(\lambda')^*\otimes T(\lambda')\otimes T(\lambda)\xrightarrow{\ev\otimes1}T(\lambda)\]
equals the projection we started with. This expresses $T(\lambda)$ as a summand of $T(\lambda')^*\otimes S$, contradicting $S\in I_n$. We have now shown that $\St_{n-1}$ is a summand of $X\otimes T(\lambda'+p^{n-1}\mu')$, and then if $Y=T(\nu)$ we set $Z'=T(\lambda'+p^{n-1}\mu')\otimes T(\nu-(p^{n-1}-1)\rho)$ so that $Y$ is a summand of $X\otimes Z'$. Finally, we take $Z=Z'\otimes Y^*\otimes Y\in J_n$, since $Y$ is a summand of $Y\otimes Y^*\otimes Y$ via evaluation and coevaluation. We have defined $Z$ over $\widetilde G$, but its weights are such that it lifts to $G$. The last statement in the lemma now follows, noting that $J_n$ contains all indecomposables in $\cT_n$ except for $\unit$.
\end{proof}

Lemma~\ref{LemThickIdeal} states that the indecomposables in $J_n\setminus I_n$ form a cell in the sense of \cite{AnCells}. The lemma is similar to \cite[Proposition~14]{AnCells}, however our proof replaces the assumption $p\geq2h-2$ with Donkin's tensor product theorem holding for $G$. Note that while $J_n$ is minimal among thick ideals in $\cT_n$ strictly containing $I_n$, this may not be true for thick ideals in $\Tilt G$ (see Question~\ref{QueWellBehaved}).

\begin{example}\label{ExSL2SL3} In the following examples, the ideals $I_n,J_n$ comprise all thick ideals of $\Tilt G$:
\begin{enumerate}
\item If $G$ is a torus, then $\cT_n=\Rep G$, $I_n$ is the zero ideal, and $J_n=\Ob(\cT_n)$.
\item We have $I_n(\SL_2)=J_{n+1}(\SL_2)$ for $n\geq1$, and by \cite[Theorem~5.3.1]{tensorideals1} the complete set of thick ideals in $\Tilt\SL_2$ is $\Ob(\Tilt\SL_2)\supset I_1\supset I_2\supset\cdots\supset0$.
\item By \cite{RaCells} and \cite{AnCells} the complete set of thick ideals in $\Tilt\SL_3$ is $\Ob(\Tilt\SL_3)=J_1\supset I_1\supset J_2\supset I_2\supset J_3\supset I_3\supset\cdots\supset0$. The first few of these are illustrated in Figure~\ref{FigSL3}.
\end{enumerate}
\end{example}

\begin{lemma}\label{LemTiltCover}
For $n\in\mN_{>0}$ and $\lambda\in X_n(T)$, the tilting module $T(2(p^n-1)\rho+w_0\lambda)$ has $L(\lambda)$ as both its socle and top, where $w_0$ is the longest element of the Weyl group $W$. That is, for $\lambda\in X(T)^+$ and $\mu\in((p^n-1)\rho+\Lambda_n)\cap X(T)$ we have
\[\dim\Hom(L(\lambda),T(\mu))=\dim\Hom(T(\mu),L(\lambda))=\begin{cases}1&\text{if }\mu=2(p^n-1)\rho+w_0\lambda\\0&\text{if }\mu\neq2(p^n-1)\rho+w_0\lambda\end{cases}.\]
\end{lemma}

\begin{proof}
Write $G_n$ for the $n$-th Frobenius kernel of $G$. As described in \cite[Proposition II.3.15, Proposition II.9.6, \S E.9]{Jantzen}, the restriction of $L(\lambda)$ to $G_nT$ is simple and its projective cover is the restriction of $T(2(p^n-1)\rho+w_0\lambda)$. A projective cover necessarily has a simple top, and restriction to $G_nT$ cannot decrease the length of the top, so $T(2(p^n-1)\rho+w_0\lambda)$ must have $L(\lambda)$ as its top over $G$ also. The duality $T(\mu)^*=T(-w_0\mu)$ gives the result for the socle.
\end{proof}

\begin{prop}\label{PropTiltCover}
The only indecomposable tilting module in $J_n\setminus I_n\subset\Ob(\cT_n)$ with $\unit$ in its socle is $T(2(p^{n-1}-1)\rho)$. Moreover, for any non-zero morphism $f:\unit\to T(2(p^{n-1}-1)\rho)$, the morphism $f\otimes\id_{\St_{n-1}}$ is split (as a morphism in $\Rep\widetilde G$ if $\St_{n-1}\not\in\Rep G$).
\end{prop}

\begin{proof}
If $n=1$, then $\St_0=\unit=T(2(p^{n-1}-1)\rho)$ and all indecomposables in $J_n\setminus I_n$ are simple, so the result is trivial. So let $n\geq2$, and let $\lambda\in(p^{n-1}-1)\rho+\Lambda_{n-1}$ and $\mu\in A$ giving an arbitrary indecomposable tilting module $T(\lambda+p^{n-1}\mu)\cong T(\lambda)\otimes L(\mu)^{(n-1)}\in J_n\setminus I_n$. Moving to the cover $\widetilde G$, there is a non-zero morphism $\unit\to T(\lambda)\otimes L(\mu)^{(n-1)}$ if and only if there is a non-zero morphism $T(\lambda)^*\to L(\mu)^{(n-1)}$ by adjunction. By Lemma~\ref{LemTiltCover} this holds if and only if
\[\lambda=2(p^{n-1}-1)\rho-p^{n-1}\mu\text{ and }p^{n-1}\mu\in\Lambda_{n-1},\]
so $\lambda+p^{n-1}\mu=2(p^{n-1}-1)\rho$ as required. By highest weight theory, $T(2(p^{n-1}-1)\rho)$ is a summand of $\St_{n-1}\otimes\St_{n-1}^*$ with multiplicity 1. Since this is the only summand with $\unit$ in its socle, the coevaluation morphism $\unit\to\St_{n-1}\otimes\St_{n-1}^*$ factors through it. This means the composition
\[\St_{n-1}\xrightarrow{f\otimes1}T(2(p^{n-1}-1)\rho)\otimes\St_{n-1}\hookrightarrow\St_{n-1}\otimes\St_{n-1}^*\otimes\St_{n-1}\xrightarrow{1\otimes\ev_{\St_{n-1}}}\St_{n-1}\]
is a non-zero multiple of $\id_{\St_{n-1}}$, and hence we have exhibited a splitting for $f\otimes\id_{\St_{n-1}}$.
\end{proof}

\begin{samepage}
\begin{prop}\label{PropUniqueIdeal}
There is exactly 1 tensor ideal in $\cT_n$ corresponding to the thick ideal $I_n$.
\end{prop}

\begin{proof}
We must show that $I_n^{\rm max}\subseteq I_n^{\rm min}$ as ideals in $\cT_n$. Let $f:X\to Y$ be a morphism in $I_n^{\rm max}$ and let $g=(f\otimes1)\coev_X$ be the corresponding morphism $\unit\to Y\otimes X^*$. If the restriction of $g$ to a summand of $Y\otimes X^*$ isomorphic to $\unit$ is non-zero, then $\id_\unit\in I_n^{\rm max}$ and $\unit\in I_n$, a contradiction. By Proposition~\ref{PropTiltCover}, the only other summands not in $I_n$ on which $g$ could be non-zero are isomorphic to $T(2(p^{n-1}-1)\rho)$. But the restriction of $g$ to such a summand would be split after tensoring with $\St_{n-1}$, meaning $\id_{\St_{n-1}}\in I_n^{\rm max}(G)$ and $\St_{n-1}\in I_n(G)$, another contradiction. Note that if $\St_{n-1}$ is not defined over $G$, we can instead use $\St_{n-1}(\widetilde G)\otimes\St_{n-1}(\widetilde G)^*$ which lifts to $G$ and is not in $I_n$ since it contains the summand $T(2(p^{n-1}-1)\rho)$. So $g$ is only non-zero on summands in $I_n$, and thus $g$ factors through an object in $I_n$, which means $f=(1\otimes\ev_X)(g\otimes1)\in I_n^{\rm min}$.
\end{proof}
\end{samepage}

\subsection{} We denote the unique ideal in Proposition~\ref{PropUniqueIdeal} by $\cI_n(G)$, or just $\cI_n$. We define
\[\Ver_{p^n}(G)=\cC(\cT_n/\cI_n,J_n/I_n)\]
using the construction in Section~\ref{SecAbEnv}. One can quickly verify that $\Ver_{p^n}(\SL_2)$ matches the definition of $\Ver_{p^n}$ from \cite{BEO}, and $\Ver_{p^1}(G)$ matches the definition of $\Ver_p(G)$ in \cite[\S 4]{asymptotic} which is based on the category ${\bf Sm}$ in \cite[Theorem~4.11]{GK}. Note that if $G$ is a torus then $\cT_n(G)/\cI_n(G)=\Rep G=\Ver_{p^n}(G)$.

\begin{theorem}\label{ThmExistence}
$\Ver_{p^n}(G)$ is an abelian envelope of $\cT_n(G)/\cI_n(G)$.
\end{theorem}

\begin{proof}
We prove this in three parts: part (1) proves the theorem for quasi-simple groups and tori, then (2) extends this to products including $\widetilde G=\widetilde{\cD G}\times T_2$ as defined in Section~\ref{SecGCover}, and then (3) with the map $\widetilde G\twoheadrightarrow G$ proves the theorem for arbitrary $G$.
\begin{enumerate}
\item If $G$ is quasi-simple, then $J_n(G)/I_n(G)$ has finitely many isomorphism classes of indecomposable objects, so $\Ver_{p^n}(G)$ is an abelian envelope of $\cT_n(G)/\cI_n(G)$ by Theorem~\ref{ThmCEOEnv}. If instead $G$ is a torus, then $\cT_n(G)/\cI_n(G)=\Rep G=\Ver_{p^n}(G)$ as described in Example~\ref{ExSL2SL3}(1). Thus the theorem holds for tori and quasi-simple $G$.
\item Now suppose $G=G_1\times G_2$ such that the theorem holds for $G_1$ and $G_2$. This means that for $i\in\{1,2\}$, all objects in $J_n(G_i)$ become either splitting or isomorphic to zero in the category $\cT_n(G_i)/\cI_n(G_i)$. By the results in Section~\ref{SecBoxTensor}, we have a full pseudo-tensor functor $\cT_n(G_1\times G_2)\to\cT_n(G_1)\dot\otimes\cT_n(G_2)$ which induces a full pseudo-tensor functor 
\[F:\cT_n(G_1\times G_2)/\cI_n(G_1\times G_2)\to(\cT_n(G_1)/\cI_n(G_1))\dot\otimes(\cT_n(G_2)/\cI_n(G_2)).\]
By \cite[Theorem~6.1.3]{quoprop}, $\Ver_{p^n}(G_1)\boxtimes\Ver_{p^n}(G_2)$ is an abelian envelope of the category on the right. It is then also an abelian envelope of the category on the left and equivalent to $\Ver_{p^n}(G_1\times G_2)$, since the essential image of $F$ contains all of the splitting objects, so the theorem holds for $G$.
\item Now suppose there is a central quotient $\phi:G'\twoheadrightarrow G$ such that the theorem holds for $G'$. In particular, $\ker\phi=\bigcap_{\lambda\in L}\ker(\lambda)$ for some lattice $L$ such that $\mZ\Phi\subseteq L\subseteq X(T')$, where $T'$ is a torus of $G'$. As an additive category, $\cT_n(G')/\cI_n(G')$ decomposes into a direct sum of additive subcategories enumerated by cosets of $L$, and $\cT_n(G)/\cI_n(G)$ is equivalent to the subcategory containing $\unit$. This is a full pseudo-tensor subcategory whose splitting objects are the images of objects in $J_n(G)$, thus $\Ver_{p^n}(G)$ is an abelian envelope by Theorem~\ref{ThmAbEnv} and the theorem holds for $G$.
\end{enumerate}
\end{proof}

Parts (2) and (3) of the proof above also give us the following useful structural results:

\begin{lemma}\label{LemVerpnTensor}
$\Ver_{p^n}(G_1\times G_2)\simeq\Ver_{p^n}(G_1)\boxtimes\Ver_{p^n}(G_2)$ for any groups $G_1,G_2$ satisfying the conditions of Section~\ref{SecAlgGrps}.
\end{lemma}

\begin{lemma}\label{LemVerpnCover}
If $N\subseteq G$ is a normal central subgroup scheme, then the induced functor $\Ver_{p^n}(G/N)\hookrightarrow\Ver_{p^n}(G)$ is an inclusion of tensor categories so that $\Ver_{p^n}(G/N)$ is a Serre subcategory of $\Ver_{p^n}(G)$.
\end{lemma}

\begin{proof}
The fact that the induced functor $\Ver_{p^n}(G/N)\to\Ver_{p^n}(G)$ is an inclusion and Serre follows from the construction of the abelian envelope as described in Section~\ref{SecAbEnv}. In particular, a simple top of a projective object in $\Ver_{p^n}(G/N)$ is sent to the simple top of the corresponding projective in $\Ver_{p^n}(G)$.
\end{proof}

Next we show that $\Ver_{p^n}(G)$ has a tensor functor to $\Ver_{p^n}$ arising from the restriction along a principal map $\SL_2\to G$. First, we provide the following extension of \cite[Lemma~3.2.5]{CEN}:

\begin{samepage}
\begin{lemma}\label{LemCharacters}
For an $\SL_2$-module $X$, write its character as a Laurent polynomial, that is
\[\ch_X(x)=\sum_{i\in\mZ}m(i)x^i\]
where $m(i)$ is the multiplicity of the weight $i\in\mZ$ in $X$. Let $X$ be a tilting module, $n\in\mN_{>0}$ and $\omega_n=\exp(\pi\sqrt{-1}/p^n)\in\mC$. Then $\ch_X(\omega_n)\geq0$, and moreover $\ch_X(\omega_n)=0$ if and only if $X\in I_n(\SL_2)$.

\end{lemma}
\end{samepage}

\begin{proof}
First suppose $X=T_a$ for some $a\in\mN$. Recall from \cite[\S 3.2.3]{CEN} that the character of the Weyl module $\Delta_a$ is the polynomial $x^a+x^{a-2}+\cdots+x^{-a}=(x^{a+1}-x^{-a-1})/(x-x^{-1})$. Using this fact and the description of tilting modules of $\SL_2$ from Section~\ref{SecSL2Tensor}, one can quickly deduce that the lemma is true for $a\leq p^2-2$. We now consider $p^n-1\leq a\leq p^{n+1}-2$ for $n\geq2$ using induction on $n$. We have
\[\begin{aligned}T_a&=T_{a'+p^{n-1}b}\otimes T_c^{(n)}\\&=T_{a'}\otimes T_b^{(n-1)}\otimes T_c^{(n)}\end{aligned}
\quad\text{for some $a',b,c$ with}\quad
\begin{array}{l}p^{n-1}-1\leq a'\leq 2p^{n-1}-2,\\p-1\leq b\leq 2p-2,\\0\leq c\leq p-2\end{array}\]
and thus $\ch_{T_a}(x)=\ch_{T_{a'}}(x)\ch_{T_b}(x^{p^{n-1}})\ch_{T_c}(x^{p^n})$. By induction we have $\ch_{T_{a'}}(\omega_{n+1})>0$, $\ch_{T_b}(\omega_{n+1}^{p^{n-1}})>0$, $\ch_{T_c}(\omega_{n+1}^{p^n})>0$ and $\ch_{T_b}(\omega_n^{p^{n-1}})=0$ giving the required result. This proves the lemma for indecomposable tiltings, and then the result for arbitrary tiltings follows from additivity and non-negativity of $\ch_X(\omega_n)$.
\end{proof}

\begin{lemma}\label{LemStImage}
Suppose $G$ is not a torus. Then the restriction of $\St_n(G)$ (or $\St_n(\widetilde G)$) to $\SL_2$ along any principal morphism $\phi:\SL_2\to G$ is in $I_n(\SL_2)\setminus I_{n+1}(\SL_2)$.
\end{lemma}

\begin{proof}
First assume $p\geq3$. Since $\St_n$ is a Weyl module, its weights are given by the Weyl character formula. Thus the restriction of $\St_n$ to $\Rep\SL_2$ has character $\ch(x)$ given by
\[\ch(x)=\sum_{w\in W}\varepsilon(w)x^{\phi^*(w(p^n\rho))}\bigg/\sum_{w\in W}\varepsilon(w)x^{\phi^*(w(\rho))}=\frac{g(x^{p^n})}{g(x)}\]
where $\varepsilon:W\to\{1,-1\}$ is the sign map, $\phi^*:X(T)\to\mZ$ is restriction of weights to $\SL_2$, and
\[g(x)=\sum_{w\in W}\varepsilon(w)x^{\phi^*(w(\rho))}=\prod_{\alpha\in\Phi^+}(x^{\phi^*(\alpha)/2}-x^{-\phi^*(\alpha)/2}).\]
Notice that $g(1)=g(-1)=0$ by the right hand formula. If we can show $g(\omega_r)\neq0$ for $r\geq1$ where $\omega_r=\exp(\pi\sqrt{-1}/p^r)$, then we will have $\ch(\omega_n)=g(-1)/g(\omega_n)=0$ and $\ch(\omega_{n+1})=g(\omega_1)/g(\omega_{n+1})\neq0$, completing the proof by Lemma~\ref{LemCharacters}. For any $\alpha\in\Phi^+$, by the results of Section~\ref{SecPrincipal} we have $0<\phi^*(\alpha)\leq 2h-2<2p$. Moreover, $\phi^*(\alpha)$ is even, since $\alpha$ is a weight of $G/Z$ and $\phi$ restricts to $\PGL_2\to G/Z$, where $Z$ is the center of $G$. Thus for $p\neq2$ and $r\geq1$, $p^r$ does not divide $\phi^*(\alpha)$ and thus $\omega_r^{\phi^*(\alpha)/2}-\omega_r^{-\phi^*(\alpha)/2}$ is non-zero. This means $g(\omega_r)\neq0$ as required.

If instead $p=2$, then the constraint $p\geq h$ means $\widetilde{\cD G}=(\SL_2)^m$ for some $m\in\mN_{>0}$. Then restriction along $\phi$ sends $\St_n(G)$ (or, if it does not exist, $\St_n(\widetilde G)$) to $\St_n(\SL_2)^{\otimes m}$. Since $\St_n(\SL_2)$ is in $I_n(\SL_2)\setminus I_{n+1}(\SL_2)$, so is $\St_n(\SL_2)^{\otimes m}$ by Lemma~\ref{LemCharacters}.
\end{proof}

\begin{prop}\label{PropRestriction1}
If $\SL_2\to G$ is a principal morphism, then the corresponding functor $\Rep G\to\Rep\SL_2$ restricts to a functor $F:\cT_n(G)\to\cT_n(\SL_2)$, and $\cI_n(G)=F^{-1}(\cI_n(\SL_2))$. This induces a tensor functor $\Ver_{p^n}(G)\to\Ver_{p^n}$.
\end{prop}

\begin{proof}
If $G$ is a torus then the result is trivial, so assume $G$ is not a torus. By the reasoning in Sections~\ref{SecPrincipal} and \ref{SecTIJDef}, we may also assume without loss of generality that $G=\widetilde{\cD G}$ is semisimple and simply-connected. The restriction of $\St_{n-1}(G)$ is in $J_n(\SL_2)$ but not in $I_n(\SL_2)$; this is trivial for $n=1$, and follows from Lemma~\ref{LemStImage} and $J_n(\SL_2)=I_{n-1}(\SL_2)$ for $n\geq2$. Every indecomposable $T(\lambda)$ in $\cT_n(G)$ aside from $\unit$ is a summand of $\St_{n-1}(G)\otimes T(\lambda-(p^{n-1}-1)\rho)$, hence the functor $\Rep G\to\Rep\SL_2$ sends $\cT_n(G)$ to $\cT_n(\SL_2)$. Moreover, the thick ideal $F^{-1}(I_n(\SL_2))$ does not contain $\St_{n-1}(G)\in J_n(G)$, so Lemma~\ref{LemThickIdeal} implies $F^{-1}(I_n(\SL_2))\subseteq I_n(G)$. Now, as described in Section~\ref{SecSL2Tensor}, \cite[Proposition~4.4]{GK} gives us $F(T(\mu))\in I_1(\SL_2)$ for $\mu\in\Lambda^+\setminus A$. Restrictions along group morphisms are monoidal and commute with Frobenius twists (for an appropriate $\mF_p$-rational structure as in Section~\ref{SecPrincipal}), so for $T(\lambda+p^{n-1}\mu)\in I_n(G)$ with $\lambda\in(p^{n-1}-1)\rho+\Lambda_{n-1}$ and $\mu\in\Lambda^+\setminus A$ we have
\[F(T(\lambda+p^{n-1}\mu))\cong F(T(\lambda))\otimes F(T(\mu))^{(n-1)}\in I_n(\SL_2).\]
Hence $I_n(G)=F^{-1}(I_n(\SL_2))$, and then Proposition~\ref{PropUniqueIdeal} gives $\cI_n(G)=F^{-1}(\cI_n(\SL_2))$. A functor $\Ver_{p^n}(G)\to\Ver_{p^n}$ then arises from the universal property of the abelian envelope.
\end{proof}

\section{A more comprehensive construction}\label{SecExpansion}

\subsection{}\label{SecBarDef} Fix an integer $n\geq1$. We define the following enlarged versions of $\cT_n$, $I_n$ and $J_n$:
\begin{enumerate}
\item Let $\overline\cT_n(G)$ be the full subcategory of $\Rep G$ with objects $X$ such that $X\otimes S\in\Tilt G$ for some $S\in J_n\setminus I_n$ (or equivalently all $S$ by Lemma~\ref{LemThickIdeal}).
\item Let $\overline I_n(G)$ be the class of objects $X$ in $\overline\cT_n(G)$ such that $X\otimes S\in I_n(G)$ for some $S\in J_n\setminus I_n$ (or equivalently all $S$ by Lemma~\ref{LemThickIdeal}).
\item Let $\overline J_n(G)$ be the class of objects in $\overline\cT_n(G)$ whose indecomposable summands are in $J_n(G)\cup\overline I_n(G)$.
\end{enumerate}
If $G$ is clear from context, we will just write $\overline\cT_n$, $\overline I_n$ and $\overline J_n$. Once again, $\overline\cT_n(G),\overline I_n(G),\overline J_n(G)$ are the restrictions/preimages of the corresponding categories/classes for $\widetilde G$ or $\widetilde{\cD G}$. Note that instead of taking $S$ arbitrary in the above definitions, we may specifically take $S$ to be $T(2(p^{n-1}-1)\rho)$ or $\St_{n-1}$ (replacing $\Tilt G$ and $I_n(G)$ with $\Tilt\widetilde G$ and $I_n(\widetilde G)$ if necessary).

\begin{lemma}~\label{LemStStable}
$\overline\cT_n$ is a Karoubi monoidal category, and $\overline I_n$ and $\overline J_n$ are thick tensor ideals in $\overline\cT_n$. Moreover, $X\otimes S\in J_n$ for $X\in\overline\cT_n$ and $S\in J_n\setminus I_n$, and every thick ideal of $\overline\cT_n$ strictly containing $\overline I_n$ also contains $\overline J_n$.
\end{lemma}

\begin{proof}
Clearly $\overline\cT_n$ is additive, and if $(X\oplus Y)\otimes S$ is a tilting module then $X\otimes S$ is a summand of a tilting module and hence tilting, so $\overline\cT_n$ is also Karoubi. Any module $S\in J_n\setminus I_n$ is a summand of $S\otimes S^*\otimes S$ via evaluation and coevaluation morphisms. So for $X,Y\in\overline\cT_n$ the module $X\otimes Y\otimes S$ is a summand of $(X\otimes S)\otimes(Y\otimes S)\otimes S^*$ and thus tilting, and if $X\otimes S$ or $Y\otimes S$ is in $\overline I_n$ then so is $X\otimes Y\otimes S$. Moreover, $S^*\in J_n\setminus I_n$ and $X^*\otimes S^*\cong (X\otimes S)^*\in\cT_n$. Hence $\overline\cT_n$ is rigid monoidal and $\overline I_n$ is a thick tensor ideal. Now, for $X\in\overline\cT_n$ and $S\in J_n$, $X\otimes S$ is a summand of $(X\otimes S)\otimes S^*\otimes S$ which is in $J_n$, so $J_n$ is a thick tensor ideal in $\overline\cT_n$ and thus $\overline J_n$ is also a thick tensor ideal. Finally, if $X\not\in\overline I_n$ and $S\in J_n\setminus I_n$, then $X\otimes S$ is in $J_n\setminus I_n$ and hence $X$ generates a thick ideal containing $J_n$ by Lemma~\ref{LemThickIdeal}.
\end{proof}

Let $\overline\cI_n(G)\coloneqq\overline I_n(G)^{\rm max}$ as a tensor ideal in $\overline\cT_n(G)$. We will write $\overline\cI_n$ when $G$ is clear from context. Lemma~\ref{LemMorFactor} describes $\overline\cI_n(G)$ more explicitly.

\begin{lemma}\label{LemMorFactor}
If $f$ is a morphism in $\overline\cT_n$, then $f\in\overline\cI_n$ if and only if $f\otimes\id_S$ factors through an object in $I_n$ for some (or equivalently all) $S\in J_n\setminus I_n$.
\end{lemma}

\begin{proof}
For $S,S'\in J_n\setminus I_n$, by Lemma~\ref{LemThickIdeal} there is a tilting module $X$ such that $S'$ is a summand of $X\otimes S$. If $f\otimes\id_S$ factors through an object in $I_n$, then so does $f\otimes\id_{S\otimes X}$ and hence so does $f\otimes\id_{S'}$, so the choice of $S$ is irrelevant. Let $\cM$ be the collection of all morphisms $f$ in $\overline\cT_n$ such that $f\otimes\id_S$ factors through an object in $I_n$ for some $S\in J_n\setminus I_n$. Then $\cM$ is a tensor ideal in $\overline\cT_n$. We have $\id_X\in\cM$ for an object $X$ if and only if $X\otimes S$ is a summand of an object in $I_n$ and hence itself is in $I_n$, thus $\Ob(\cM)=\overline I_n$ and $\cM\subseteq\overline\cI_n$. Now consider a morphism $f\not\in\cM$, meaning $f\otimes\id_S\not\in I_n^{\rm min}=\cI_n=I_n^{\rm max}$ by Proposition~\ref{PropUniqueIdeal}. This means that for any tensor ideal in $\cT_n$ containing $f\otimes\id_S$, the corresponding thick tensor ideal contains an object in $\Ob(\cT_n)\setminus I_n$. That is, $f\otimes\id_S$ generates the identity morphism on this object, hence $f\not\in\overline\cI_n$. Thus we have shown $f\in\cM$ if and only if $f\in\overline\cI_n$ as required.
\end{proof}

\begin{theorem}\label{ThmFullEnv}
$\overline\cT_n(G)/\overline\cI_n(G)$ has abelian envelope $\Ver_{p^n}(G)$.
\end{theorem}

\begin{proof}
The indecomposables in $\overline J_n/\overline I_n\subseteq\overline\cT_n/\overline\cI_n$ are the same as those in $J_n/I_n\subseteq\cT_n/\cI_n$, as are the morphisms between them. Thus $\cC(\cT_n/\cI_n,J_n/I_n)\simeq\cC(\overline\cT_n/\overline\cI_n,\overline J_n/\overline I_n)$, and we simply need to show that this is an abelian envelope. First, $\overline\cT_n(\cD G)/\overline\cI_n(\cD G)$ has an abelian envelope by Theorem~\ref{ThmCEOEnv} and this implies that all objects in $\overline J_n(\cD G)$ are splitting objects in this category. Let $\cC=\overline\cT_n(\cD G\times T_2)/\overline\cI_n(\cD G\times T_2)$, which is equivalent to $(\overline\cT_n(\cD G)/\overline\cI_n(\cD G))\dot\otimes\Rep T_2$ where $\dot\otimes$ is defined in Section~\ref{SecBoxTensor}. The objects in $\overline J_n(\cD G\times T_2)$ are splitting objects in $\cC$, and hence also splitting objects in $\overline\cT_n(G)/\overline\cI_n(G)$ which is a full subcategory of $\cC$. Moreover, $\overline\cT_n(G)/\overline\cI_n(G)$ is separated by Lemma~\ref{LemMorFactor}, so the result now follows from Theorem~\ref{ThmAbEnv}.
\end{proof}

We recall the conjecture [DFilt $\Rightarrow$] from \cite{dconj3} which states that $\St_{n-1}\otimes L(\lambda)$ is a tilting module (that is $L(\lambda)\in\overline\cT_n$) for $\lambda\in X_{n-1}(T)$. Similarly to Donkin's tensor product theorem, this conjecture is proven for $p\geq 2h-4$ in \cite{dconj3}, and the only known counterexamples have characteristic $p<h$.

\begin{prop}\label{PropSimps}
Suppose $L(\lambda)\in\overline\cT_n(\widetilde G)$ for all $\lambda\in\Lambda_{n-1}$. Then $L(\lambda+p^{n-1}\mu)\in\overline\cT_n$ for all $\lambda\in\Lambda_{n-1}$ and $\mu\in A$ with $\lambda+p^{n-1}\mu\in X(T)$, and the images of these modules in $\Ver_{p^n}(G)$ are a complete set of representatives of isomorphism classes of simple modules in $\Ver_{p^n}(G)$. Consequently, we have a Steinberg tensor product theorem in $\Ver_{p^n}(G)$ as in Theorem~\ref{ThmMain}(\ref{ItemSimps}).
\end{prop}

\begin{proof}
Moving to the cover $\widetilde G$, we have $L(\lambda+p^{n-1}\mu)\cong L(\lambda)\otimes L(\mu)^{(n-1)}$ by Steinberg's tensor product theorem. Since $L(\lambda)\otimes\St_{n-1}\in J_n(\widetilde G)$ by Lemma~\ref{LemStStable}, we can apply Donkin's theorem to each summand of $L(\lambda)\otimes\St_{n-1}$ tensored with $L(\mu)^{(n-1)}$ to show that $L(\lambda+p^{n-1}\mu)\otimes\St_{n-1}$ is a tilting module. By Lemma~\ref{LemTiltCover}, $L(\lambda)\otimes L(\mu)^{(n-1)}$ is in the socle of
\[T(2(p^{n-1}-1)\rho+w_0\lambda)\otimes L(\mu)^{(n-1)}=T(2(p^{n-1}-1)\rho+w_0\lambda+p^{n-1}\mu).\]
In particular, for each indecomposable $X\in J_n\setminus I_n$, $\Hom_G(X,L(\lambda+p^{n-1}\mu))$ is 1-dimensional for exactly one weight $\lambda+p^{n-1}\mu$ and zero for all others. Thus, in the abelian envelope construction in Section~\ref{SecAbEnv}, the image of $L(\lambda+p^{n-1}\mu)$ is at most 1-dimensional as a $C$-comodule. Hence it is either simple or zero, but it cannot be zero, since $L(\lambda+p^{n-1}\mu)\otimes\St_{n-1}$ contains a summand isomorphic to $T(\lambda+p^{n-1}\mu+(p^{n-1}-1)\rho)$ which is not in $I_n$. So we have found simple socles of all projectives in $\Ver_{p^n}(G)$. Finally, the tensor product theorem is a consequence of the fact that the functor $\overline\cT_n\to\Ver_{p^n}(G)$ is monoidal.
\end{proof}

\begin{prop}\label{PropRestriction2}
If $\SL_2\to G$ is a principal morphism, then the corresponding functor $\Rep G\to\Rep\SL_2$ restricts to a functor $F:\overline\cT_n(G)\to\overline\cT_n(\SL_2)$, and $\overline\cI_n(G)=F^{-1}(\overline\cI_n(\SL_2))$. This commutes with the tensor functor $\Ver_{p^n}(G)\to\Ver_{p^n}$.
\end{prop}

\begin{proof}
Let $S\in J_n(G)\setminus I_n(G)$ and $X\in\overline\cT_n(G)$. By Proposition~\ref{PropRestriction1}, $F(S)\in J_n(\SL_2)\setminus I_n(\SL_2)$ and $F(X)\otimes F(S)\cong F(X\otimes S)\in\cT_n(\SL_2)$, hence $F(X)\in\overline\cT_n(\SL_2)$. Moreover $X\otimes S\in I_n(G)$ if and only if $F(X\otimes S)\in I_n(\SL_2)$, so $F^{-1}(\overline I_n(\SL_2))=\overline I_n(G)$ and $F^{-1}(\overline\cI_n(\SL_2))\subseteq\overline\cI_n(G)$. Now let $f$ be a morphism in $\overline\cI_n(G)$, so $f\otimes\id_S$ factors through an object in $I_n(G)$ by Lemma~\ref{LemMorFactor}. Thus $F(f)\otimes F(\id_S)$ factors through an object in $I_n(\SL_2)$, so by Lemma~\ref{LemMorFactor} again we have $F(f)\in\overline\cI_n(\SL_2)$ giving the other inclusion $\overline\cI_n(G)\subseteq F^{-1}(\overline\cI_n(\SL_2))$.
\end{proof}

Although $I_n(G)$ is a thick ideal in $\Tilt G$, we have primarily been considering it as a thick ideal in $\cT_n(G)$. This is because the correct thick ideal to consider in $\Tilt G$ in order to obtain $\Ver_{p^n}(G)$ is actually $\overline I_n(G)\cap\Ob(\Tilt G)$, as shown in the following proposition. A priori, this may be strictly larger than $I_n(G)$, see Question~\ref{QueWellBehaved}.

\begin{prop}\label{PropWellBehaved}
Let $F:\Tilt G\to\Tilt\SL_2$ be restriction to a principal $\SL_2$. We have
\[F^{-1}(\cI_n(\SL_2))=(F^{-1}(I_n(\SL_2)))^{\rm max}=\overline\cI_n(G)\cap\Mor(\Tilt G)=(\overline I_n(G)\cap\Ob(\Tilt G))^{\rm max}\]
and the abelian envelope of $\Tilt G/F^{-1}(\cI_n(\SL_2))$ is $\Ver_{p^n}(G)$.
\end{prop}

\begin{proof}
Let $\cM=\overline\cI_n(G)\cap\Mor(\Tilt G)$. The functor $H:\overline\cT_n(G)\to\overline\cT_n(\SL_2)$ restricts to $F$, so
\begin{align*}
\Ob(\cM)=H^{-1}(\overline I_n(\SL_2))\cap\Ob(\Tilt G)&=F^{-1}(\overline I_n(\SL_2)\cap\Ob(\Tilt\SL_2))=F^{-1}(I_n(\SL_2)),\\
\cM=H^{-1}(\overline\cI_n(\SL_2))\cap\Mor(\Tilt G)&=F^{-1}(\overline\cI_n(\SL_2)\cap\Mor(\Tilt\SL_2))=F^{-1}(\cI_n(\SL_2))
\end{align*}
by Proposition~\ref{PropRestriction2} and the classification of tensor ideals in $\Tilt\SL_2$. It now suffices to show that $\Ob(\cM)^{\rm max}\subseteq\cM$. If $f\in\Mor(\Tilt G)$ with $f\not\in\cM$, then $f\otimes\id_S$ does not factor through an object in $I_n(G)$ for any $S\in J_n\setminus I_n$ by Lemma~\ref{LemMorFactor}. This means that $f\otimes\id_S\not\in\cI_n(G)$ by Proposition~\ref{PropUniqueIdeal} and thus $f$ generates an object in $\cT_n(G)$ that is not in $\Ob(\cM)$, so $f\not\in\Ob(\cM)^{\rm max}$ as required. For the abelian envelope, we have an inclusion $\Tilt G/\cM\to\overline\cT_n(G)/\overline\cI_n(G)$. Objects in $J_n(G)$ are splitting in the second category by the proof of Theorem~\ref{ThmFullEnv}, so they are also splitting in $\Tilt G/\cM$ and we can apply Theorem~\ref{ThmAbEnv}.
\end{proof}

This shows that the definition of $\Ver_{p^n}(G)$ in Theorem~\ref{ThmMain} is equivalent to the definition in Section~\ref{SecDefinition}. Property (\ref{ItemProj}) in that theorem follows from the construction, while (\ref{ItemAbEnv}) is Theorem~\ref{ThmFullEnv}, (\ref{ItemSimps}) is Proposition~\ref{PropSimps}, and (\ref{ItemPrin}) is Proposition~\ref{PropRestriction2}.

In addition to $I_n$ and $\overline I_n\cap\Ob(\Tilt G)$, we recall another family of ideals in $\Tilt G$ defined in \cite{HW}: the tensor ideals $\cN_k$ and thick tensor ideals $N_k$ for $k\in\mN_{>0}$ generalise the ideals of negligible morphisms and objects ($\cN_1=\cI_1$ and $N_1=I_1$). For $p>h$, by \cite[Corollary~7.13, Proposition~8.8]{HW} and Lemma~\ref{LemThickIdeal}, the restriction of $N_{(n-1)|\Phi^+|+1}$ to $\cT_n$ is equal to $I_n$, and then by Proposition~\ref{PropUniqueIdeal} the restriction of $\cN_{(n-1)|\Phi^+|+1}$ to $\cT_n$ is equal to $\cI_n$. A natural question is whether these ideals are also equal in $\Tilt G$:

\begin{question}\label{QueWellBehaved}
For which simple $G$ and $p$ are the following statements true?
\begin{enumerate}
\item[(a)] $I_n(G)=\overline I_n(G)\cap\Ob(\Tilt G)$;
\item[(b)] $I_n(G)=N_{(n-1)|\Phi^+|+1}$ in $\Tilt G$;
\item[(c)] $\overline\cI_n(G)\cap\Mor(\Tilt G)=\cN_{(n-1)|\Phi^+|+1}$ in $\Tilt G$.
\end{enumerate}
\end{question}

Equivalent statements to (a) are that $\St_{n-1}\otimes T(\lambda)\not\in I_n(\widetilde G)$ for any indecomposable tilting $T(\lambda)\not\in I_n(G)$, or that every thick tensor ideal in $\Tilt G$ strictly containing $I_n(G)$ also contains $J_n(G)$. All three statements have an affirmative answer for $\SL_2$ with $p\geq2$, and $\SL_3$ with $p\geq3$, by the classification of thick tensor ideals described in Example~\ref{ExSL2SL3}. In a previous version of this paper, Question~\ref{QueWellBehaved} was stated as a conjecture proposing that all three statements were true for all $G$. The author has since been made aware of computational evidence that (a) is not satisfied in type $B_2=C_2$, although this is work-in-progress and not presented in this paper.

For $\SL_3$ we also have the following result regarding tensor ideals:

\begin{prop}
For $p\geq3$, there is exactly 1 tensor ideal in $\Tilt\SL_3$ corresponding to the thick tensor ideal $I_2(\SL_3)$.
\end{prop}

\begin{proof}
Suppose $f:X\to Y$ is a morphism in $I_2^{\rm max}$ (as an ideal in $\Tilt\SL_3$). If the corresponding adjoint map $g:\unit\to Y\otimes X^*$ factors through an indecomposable object in $J_2\setminus I_2$, then we reach a contradiction by the proof of Proposition~\ref{PropUniqueIdeal}, so suppose it factors through an object not in $J_2$. By \cite[Theorem~1.3]{RaCells}, any $T(\lambda)\not\in J_2$ has the same Weyl factor multiplicities as the corresponding indecomposable tilting module of $U_q(\mathfrak{sl}_3)$ with $q$ a $p$-th root of unity, so by \cite[\S 6.2.1]{tensorideals2} the only non-trivial such modules $T(\lambda)$ with $\unit$ in their socle are $T((p-2)\rho)$, $T((3p-3)\varpi_1)$, $T((3p-3)\varpi_2)$. Let $X$ be one of these modules. For each simple composition factor $L(\nu)$ of $X$, $\nu$ is small enough that $\St_1\otimes L(\nu)$ is tilting. Specifically, we have $L(\nu)\cong L(\lambda)\otimes L(\mu)^{(1)}$ where $\lambda\in X_1(T)$ and $L(\mu)=T(\mu)$ (in particular, the only case with $\mu\not\in\overline A$ is $p=3$ and $\nu=6\varpi_i$ for $i\in\{1,2\}$, for which $\mu=2\varpi_i$ satisfies $L(\mu)=T(\mu)$), so we can apply \cite[Theorem~1.2.1]{dconj3} to $L(\lambda)$ and Donkin's tensor product theorem to $L(\mu)^{(1)}$. So all composition factors of $X$ become tilting when tensored with $\St_1$, and then since tilting modules have no extensions with each other by \cite[Corollary~E.2]{Jantzen}, the short exact sequence
\[0\to\St_1\to X\otimes\St_1\to X/\unit\otimes\St_1\to0\]
splits. Thus the inclusion $\unit\to X$ being in $I_2^{\rm max}$ implies $\St_1\in I_2$, a contradiction. Hence, if $g$ is a morphism in $I_2^{\rm max}$, then it is also in $I_2^{\rm min}$ as required.
\end{proof}

\begin{question}
Does $I_n$ have a unique corresponding tensor ideal in $\Tilt G$ for all $G$ and $p$?
\end{question}

\begin{remark}\label{RemNonFull}
The functor $\overline\cT_n/\overline\cI_n\to\Ver_{p^n}(G)$ is typically not full. In particular, it is possible for non-isomorphic objects in $\overline\cT_n/\overline\cI_n$ to have isomorphic images in $\Ver_{p^n}(G)$. For example, consider the case $G=\SL_2$, $p=3$ and $n=2$. The indecomposable tiltings up to weight $p^n-2=7$ have the following diagrams of composition factors (in the sense of \cite{Alperin}):
\[L_0,\quad L_1,\quad L_2,\quad
\begin{tikzpicture}[baseline=-0.65ex,scale=0.7]
\node (1) at (0,-1) {$L_1$}; \node (2) at (0,0) {$L_3$}; \node (3) at (0,1) {$L_1$};
\draw (1) -> (2); \draw (2) -> (3); \end{tikzpicture},\quad
\begin{tikzpicture}[baseline=-0.65ex,scale=0.7]
\node (1) at (0,-1) {$L_0$}; \node (2) at (0,0) {$L_4$}; \node (3) at (0,1) {$L_0$};
\draw (1) -> (2); \draw (2) -> (3); \end{tikzpicture},\quad L_5,\quad
\begin{tikzpicture}[baseline=-0.65ex,scale=0.7]
\node (1) at (0,-1) {$L_4$}; \node (2) at (-0.8,0) {$L_6$}; \node (3) at (0.8,0) {$L_0$}; \node (4) at (0,1) {$L_4$};
\draw (1) -> (2); \draw (1) -> (3); \draw (3) -> (4); \draw (2) -> (4); \end{tikzpicture},\quad
\begin{tikzpicture}[baseline=-0.65ex,scale=0.7]
\node (1) at (0,-1) {$L_3$}; \node (2) at (-0.8,0) {$L_7$}; \node (3) at (0.8,0) {$L_1$}; \node (4) at (0,1) {$L_3$};
\draw (1) -> (2); \draw (1) -> (3); \draw (3) -> (4); \draw (2) -> (4); \end{tikzpicture}.\]
The process of translating this picture to $\Ver_9$ can be intuited as ``removing every instance of $L_6$ and $L_7$'', and in Section~\ref{SecVerpnSL2} we formalise this idea as a Serre quotient. This changes the filtrations of $T_6$ and $T_7$, becoming $[L_4,L_0,L_4]$ and $[L_3,L_1,L_3]$ in $\Ver_9$. However, notice that the subobjects of $T_6$ given by $L_4$ and $[L_4,L_6]$ are non-isomorphic objects in $\overline\cT_2/\overline\cI_2$, but they both have image $L_4$ in $\Ver_9$.

The following is another example in the case $G=\SL_3$, $p=3$ and $n=2$. We write weights $k_1\varpi_1+k_2\varpi_2$ as $(k_1,k_2)$, and note that $\St_1=L(2,2)$. It can be computed that
\[\St_1\otimes T(1,0)^*\cong T(2,3),\quad T(1,2)\otimes T(1,0)\cong\St_1\oplus T(0,3),\quad T(8,0)\otimes T(1,0)\cong T(9,0)\oplus\St_1,\]
and the only indecomposable object of $J_2\setminus I_2$ in the linkage class of $T(1,2)$ or $T(8,0)$ is $T(2,3)$. Hence, in the construction of the abelian envelope described in Section~\ref{SecAbEnv}, the images in $\Ver_9(\SL_3)$ of $T(1,2)$ and $T(8,0)$ are the spaces
\begin{align*}
\Hom(T(2,3),T(1,2))&\cong\Hom(\St_1,\St_1\oplus T(0,3)),\\
\Hom(T(2,3),T(8,0))&\cong\Hom(\St_1,T(9,0)\oplus\St_1)
\end{align*}
respectively, which are both 1-dimensional. Since $T(2,3)$ is indecomposable, $\End(T(2,3))$ has a basis $\{1,f_1,\dots,f_m\}$ with each $f_i$ nilpotent, and thus the action of $f_i$ on any 1-dimensional module of $\End(T(2,3))$ is zero. This means the morphism sets above are isomorphic as $\End(T(2,3))$-modules, and so $T(1,2)$ and $T(8,0)$ have isomorphic images in $\Ver_9(\SL_3)$.
\end{remark}

\begin{remark}\label{RemQuantum}
In \cite{STWZ}, higher Verlinde categories for the quantum group of $\SL_2$ are defined, and their properties are studied extensively in \cite{decoppet}. We note here that this definition naturally extends to other groups $G$.

We assume for simplicity that $p\geq 2h-2$ and $G=\widetilde{\cD G}$ is quasi-simple and simply connected. Let $l\geq h$ be an odd integer, not divisible by 3 if $G$ has type $G_2$, and not divisible by $p$. Write $U_q$ for the divided-power quantum group corresponding to $G$ for a primitive $l$-th root of unity $q$ in $\bk$. Write $L_q(\lambda)$ and $T_q(\lambda)$ for the simple and indecomposable tilting modules of $U_q$ with highest weight $\lambda\in\Lambda^+$, and we define a quantum Steinberg module $\St_q=L_q((l-1)\rho)=T_q((l-1)\rho)$. For $n\in\mN_{>0}$, write
\[\Lambda_{n,q}=\{\lambda\in\Lambda^+\mid\langle\lambda,\alpha_i^\vee\rangle<lp^{n-1}\text{ for }1\leq i\leq r\}\]
where $r$ is the rank of $G$ and $\theta=\theta_j$ is defined in Section~\ref{SecGCover}. We recall from \cite{pfilt1} that there is a functor $\Rep G\to\Rep U_q$ which we denote $X\mapsto X^{[l]}$, and \cite[Corollary~5.8]{pfilt1} states
\[T_q(\lambda+l\mu)\cong T_q(\lambda)\otimes T(\mu)^{[l]}\text{ for $\lambda\in(l-1)\rho+\Lambda_{1,q}$ and $\mu\in\Lambda^+$}.\]
Consequently, the functor $\Tilt G\to\Tilt U_q$ given by $X\mapsto\St_q\otimes X^{[l]}$ is injective on isomorphism classes of objects. For an integer $n\geq2$, we make the following definitions:
\begin{enumerate}
\item Let $\cT_n(U_q)$ be the full subcategory of $\Tilt U_q$ consisting of objects whose indecomposable summands have highest weights in $\{0\}\cup((lp^{n-2}-1)\rho+\Lambda^+)$.
\item Let $I_n(U_q)$ be the class of objects in $\Tilt U_q$ whose indecomposable summands have highest weights in $(lp^{n-2}-1)\rho+\Lambda_{n-1,q}+lp^{n-2}(\Lambda^+\setminus A)$.
\item Let $J_n(U_q)$ be the class of objects in $\Tilt U_q$ whose indecomposable summands have highest weights in $(lp^{n-2}-1)\rho+\Lambda^+$.
\end{enumerate}
For $n=1$ we instead define $\cT_1(U_q)=\Tilt U_q$ and $J_1(U_q)=\Ob(\Tilt U_q)$, and set $I_1(U_q)$ to be the thick ideal with indecomposables $T_q(\lambda)$ for $\lambda\in\Lambda^+\setminus A_q$, where $A_q$ is the fundamental alcove with respect to the $l$-scaled action of the affine Weyl group. We can then define $\overline\cT_n(U_q)$, $\overline I_n(U_q)$ and $\overline J_n(U_q)$ similarly to in Section~\ref{SecBarDef}.

For $\lambda\in(l-1)\rho+\Lambda_{1,q}$, $T_q(\lambda)$ is a summand of $\St_q\otimes T_q(\lambda-(l-1)\rho)$. So by adjunction we obtain a morphism $\St_q\hookrightarrow T_q(\lambda)\otimes T_q(\lambda-(l-1)\rho)^*$, and the weights of the latter module are at most $3(l-1)\rho$. By \cite[Lemma~5.6]{pfilt1}, $\St_q$ is injective among modules with weights $\mu$ satisfying $\langle\mu,\theta^\vee\rangle+(l+1)(h-1)<2lp$ where $\theta$ is the highest short root, and this includes $\mu\leq 3(l-1)\rho$ since
\[\langle3(l-1)\rho,\theta^\vee\rangle+(l+1)(h-1)=(4l-2)(h-1)=(2l-1)(2h-2)\leq (2l-1)p<2lp.\]
Thus the map $\St_q\hookrightarrow T_q(\lambda)\otimes T_q(\lambda-(l-1)\rho)^*$ is split, and so all $T_q(\lambda)$ for $\lambda\in(l-1)\rho+\Lambda_{1,q}$ generate $\St_q$ in the tensor ideal sense. This, along with the tensor product theorem above, shows that a version of \cite[Proposition~14]{AnCells} holds for $U_q$, and thus Lemma~\ref{LemThickIdeal} applies to $U_q$ as well as $G$. Thus $J_n(U_q)$ is the minimal thick tensor ideal strictly containing $I_n(U_q)$ in $\cT_n(U_q)$, and so an abelian envelope $\Ver_{p^{(n)}}^q(G)$ of $\cT_n(U_q)/I_n(U_q)^{\rm max}$ exists by Theorem~\ref{ThmCEOEnv}. Note that in this case the commutor $X\otimes Y\xrightarrow{\sim} Y\otimes X$ is a non-symmetric braiding, and $\Ver_{p^{(n)}}^q(G)$ is a braided tensor category. Repeating the proof of Lemma~\ref{LemStStable}, we can also conclude that the abelian envelope of $\overline\cT_n(U_q)/\overline I_n(U_q)^{\rm max}$ is $\Ver_{p^{(n)}}^q(G)$.
\end{remark}

\section{Properties of $\Ver_{p^n}(G)$}\label{SecProperties}

We now show that the functor $\Ver_{p^n}\to\Ver_{p^{n+1}}$ from \cite[\S 4.10]{BEO} arises directly from the Frobenius twist functor on $\Rep\SL_2$, and that this generalises to arbitrary $G$, proving property (\ref{ItemFrob}) in Theorem~\ref{ThmMain}.

\begin{lemma}\label{LemFrob}
For $n\in\mN_{>0}$ the Frobenius twist functor $(-)^{(1)}:\Rep G\to\Rep G$ restricts to a functor $F:\overline\cT_n(G)\to\overline\cT_{n+1}(G)$ with $F^{-1}(\overline\cI_{n+1}(G))=\overline\cI_n(G)$.
\end{lemma}

\begin{proof}
We have $X\in\overline\cT_n(G)$ if and only if $X\otimes\St_{n-1}$ is a tilting module of $\widetilde G$. Then by Donkin's tensor product theorem, $X^{(1)}\otimes\St_n\cong(X\otimes\St_{n-1})^{(1)}\otimes\St_1$ is also a tilting module, so $X^{(1)}\in\overline\cT_{n+1}(G)$. Consider an indecomposable summand $Y\subseteq X\otimes\St_{n-1}$, which is isomorphic to $T((p^{n-1}-1)\rho+\lambda+p^{n-1}\mu)$ for some $\lambda\in\Lambda_{n-1}$ and $\mu\in\Lambda^+$ by Lemma~\ref{LemStStable}. This corresponds to a summand of $X^{(1)}\otimes\St_n$ given by
\[Y^{(1)}\otimes\St_1\cong T((p-1)\rho+p((p^{n-1}-1)\rho+\lambda+p^{n-1}\mu))=T((p^n-1)\rho+p\lambda+p^n\mu).\]
Now $Y^{(1)}\otimes\St_1\in I_{n+1}(\widetilde G)$ if and only if $\mu\not\in A$ if and only if $Y\in I_n(\widetilde G)$. By considering all summands of $X\otimes\St_{n-1}$, we have $X\in\overline I_n(G)$ if and only if $X^{(1)}\in\overline I_{n+1}(G)$, and hence $F^{-1}(\overline I_{n+1}(G))=\overline I_n(G)$. If $f\in\overline\cI_n(G)$ then $f\otimes\id_{\St_{n-1}}$ factors through an object in $I_n(\widetilde G)$ by Lemma~\ref{LemMorFactor}, so $f^{(1)}\otimes\id_{\St_n}=(f\otimes\id_{\St_{n-1}})^{(1)}\otimes\id_{\St_1}$ factors through an object in $I_{n+1}(\widetilde G)$ meaning $f^{(1)}\in\overline\cI_{n+1}(G)$, hence $F^{-1}(\overline\cI_{n+1}(G))=\overline\cI_n(G)$.
\end{proof}

\begin{prop}\label{PropFrob}
There is an inclusion of tensor categories $\Ver_{p^n}(G)\hookrightarrow\Ver_{p^{n+1}}(G)$ for $n\in\mN_{>0}$ such that the diagram of functors
\[\begin{tikzcd}
\overline\cT_n(G) \arrow{d}{(-)^{(1)}} \arrow{r} & \overline\cT_n(G)/\overline\cI_n(G) \arrow{d} \arrow{r} & \Ver_{p^n}(G) \arrow[hook]{d} \arrow{r} & \Ver_{p^n} \arrow[hook]{d}\\
\overline\cT_{n+1}(G) \arrow{r} & \overline\cT_{n+1}(G)/\overline\cI_{n+1}(G) \arrow{r} & \Ver_{p^{n+1}}(G) \arrow{r} & \Ver_{p^{n+1}}
\end{tikzcd}\]
commutes, where the functors $\Ver_{p^n}(G)\to\Ver_{p^n}$ and $\Ver_{p^{n+1}}(G)\to\Ver_{p^{n+1}}$ come from a fixed principal morphism $\SL_2\to G$ compatible with the Frobenius morphisms on $G$ and $\SL_2$.
\end{prop}

\begin{proof}
The functor $F':\Ver_{p^n}(G)\to\Ver_{p^{n+1}}(G)$ comes from the functor $F$ in Lemma~\ref{LemFrob}. This commutes with $\Ver_{p^n}\to\Ver_{p^{n+1}}$ via the functors in Proposition~\ref{PropRestriction2} since the Frobenius twist commutes with restriction to a principal $\SL_2$. It remains to be shown that $F'$ is an inclusion of tensor categories. This is trivial if $G$ is a torus, and then by Lemmas~\ref{LemVerpnTensor} and \ref{LemVerpnCover} applied to the isogeny $\cD G\times T_2\twoheadrightarrow G$, it suffices to prove that $F'$ is an inclusion for quasi-simple $G$. This means we may assume that $\Ver_{p^n}(G)$ is a finite tensor category, so by \cite[Proposition~1.8.19]{EGNO}, the full subcategory ${\sf Im}F'$ of $\Ver_{p^{n+1}}(G)$ consisting of subquotients of objects in the image of $F'$ is also a finite tensor category. We shall show that $F'$ induces an equivalence on the subcategories of projective objects in $\Ver_{p^n}(G)$ and ${\sf Im}F'$, showing that $\Ver_{p^n}(G)\simeq{\sf Im}F'$ and $F'$ is an inclusion of tensor categories.

By \cite[Theorem~6.1.16]{EGNO}, $F'$ sends projectives in $\Ver_{p^n}(G)$ to projectives in ${\sf Im}F'$. Let $T(\lambda)$ be an indecomposable object in $J_n(G)\setminus I_n(G)$, and let
\[V\coloneqq\Hom_{\Ver_{p^{n+1}}(G)}(\unit,F'(T(\lambda)))\cong\Hom_{\Ver_{p^{n+1}}(G)}(\unit,T(\lambda)^{(1)}).\]
To show $F'$ is full on projective objects, by adjunction and Proposition~\ref{PropTiltCover} it suffices to show that $\dim V\leq1$ for $\lambda=2(p^{n-1}-1)\rho$ and 0 otherwise. Since $\unit$ is a quotient of $T(2(p-1)\rho)^*$, $\dim V\leq\dim V'$ where
\[V'\coloneqq\Hom_{\Ver_{p^{n+1}}(G)}(T(2(p-1)\rho)^*,T(\lambda)^{(1)}).\]
By adjunction, Donkin's tensor product theorem, and the fact that the modules $T(2(p-1)\rho)$ and $T(\lambda)^{(1)}$ are both in $\overline\cT_{n+1}(G)$,
\[V'\cong\Hom_{\Ver_{p^{n+1}}(G)}(\unit,T(2(p-1)\rho+p\lambda)).\]
Then using \cite[Proposition~2.24]{BEO} we have
\[V'\cong\Hom_{\cT_{n+1}(G)/\cI_{n+1}(G)}(\unit,T(2(p-1)\rho+p\lambda))\]
which has the required dimension by Proposition~\ref{PropTiltCover}. So $F'$ is full on projectives.

To show essential surjectivity to projectives in ${\sf Im}F'$, suppose $Q\in{\sf Im}F'$ is projective, and hence also injective by \cite[Proposition~6.1.3]{EGNO}. By definition of ${\sf Im}F'$ there exists $X\in\Ver_{p^n}(G)$ such that $Q$ is a subquotient of $F'(X)$, and then $X$ is a subquotient of some projective $P$ since $\Ver_{p^n}(G)$ is finite. Objects that are projective and injective have no extensions, so $Q$ splits from $F'(P)$ as a summand. Then by fullness and additivity of $F'$, $Q$ has a preimage in $\Ver_{p^n}(G)$ which is a summand of $P$ and hence projective.
\end{proof}

\subsection{The subcategory $\Ver_{p^n}^+(G)$} Write $Z=Z(G)$ for the centre of $G$, and $G^{\rm ad}=G/Z$. Note that $G^{\rm ad}\cong\widetilde G/Z(\widetilde G)\cong\cD G/Z(\cD G)$. We have induced functors
\[\Rep G^{\rm ad}\to\Rep G,\quad \Tilt G^{\rm ad}\to\Tilt G,\quad \cT_n(G^{\rm ad})\to\cT_n(G)\]
which are full pseudo-tensor functors, and their essential images consist of all modules whose weights are in the root lattice. We see from Proposition~\ref{PropUniqueIdeal} that the preimage of $\cI_n(G)$ under the latter functor is $\cI_n(G^{\rm ad})$. Hence we obtain an inclusion of tensor categories $\Ver_{p^n}(G^{\rm ad})\hookrightarrow\Ver_{p^n}(G)$, and this is an equivalence with a tensor subcategory of $\Ver_{p^n}(G)$ which we denote $\Ver_{p^n}^+(G)$. For example, the category $\Ver_{p^n}^+=\Ver_{p^n}^+(\SL_2)$ as defined in \cite{BEO} is equivalent to $\Ver_{p^n}(\PGL_2)$.

Fix a principal map $\phi:\SL_2\to G$. As described in \cite[\S 3.3.1]{CEN}, this induces a map $z:\mZ/2\to Z$ between the centres. Recall from \cite[Theorem 3.3.2]{CEN} that if $p>h$ then the subcategory of $\Ver_p(G)$ spanned by the invertible objects is equivalent to $\Rep_{\sVec}(Z,z)$ (this is stated for quasi-simple groups, but the result naturally extends to connected reductive groups). The following proposition is property (\ref{ItemDecomp}) in Theorem~\ref{ThmMain}:

\begin{prop}\label{PropCentDecomp}
If $p>h$ (in addition to the assumptions in Section~\ref{SecAlgGrps}), then
\[\Ver_{p^n}(G)\simeq\Ver_{p^n}^+(G)\boxtimes\Rep_{\sVec}(Z,z).\]
\end{prop}

\begin{proof}
We assume $G$ is semisimple, and then Lemmas~\ref{LemVerpnTensor} and \ref{LemVerpnCover} extend the result to arbitrary $G$. By \cite[Proposition~3.3.6]{CEN}, the invertible objects in $\Ver_p(G)$ come from tilting modules $T(\mu)$ whose highest weights $\mu\in A$ are a complete set of coset representatives for $X(T)/\mZ\Phi$, where $\mZ\Phi$ is the root lattice. Then $T(\mu)^{(n-1)}\in\Ver_{p^n}(G)$ is the image of $T(\mu)$ under $\Ver_p(G)\to\Ver_{p^n}(G)$, and hence is also invertible. Since the weights of an indecomposable tilting module $T(\lambda)$ are in $\lambda+\mZ\Phi$, we can decompose $\Ver_{p^n}(G)$ (as an abelian category) into a direct sum of subcategories consisting of modules whose weights lie in a particular $\mZ\Phi$-coset. Then tensoring by $T(\mu)^{(n-1)}$ permutes these subcategories in the same way that $p^{n-1}\mu$ acts additively on $X(T)/\mZ\Phi$, and $\Ver_{p^n}^+(G)$ is exactly the subcategory containing $\unit$, so we have $\Ver_{p^n}(G)\simeq\Ver_{p^n}^+(G)\boxtimes\Rep_{\sVec}(Z,z)$.
\end{proof}

Note that Proposition~\ref{PropCentDecomp} fails for $p=h$. This is demonstrated for $G=\SL_2$ with $p=2$ in \cite{BEO}, and it can also be seen in the case $G=\SL_3$ with $p=3$.

\subsection{Perfection and $\Ver_{p^\infty}(G)$}\label{SecPerf} For a $\bk$-module $V$ and $r\in\mZ$ we write $V^{(-r)}$ for the $\bk$-module with the same underlying set as $V$ but with $\bk$-action $\lambda\cdot v=\lambda^{p^r}v$ for $\lambda\in\bk$ and $v\in V$ (see \cite[\S F.1]{Jantzen}). If $A$ is a $\bk$-algebra, then $A^{(-r)}$ is also a $\bk$-algebra, and we have a Frobenius morphism $\Fr:A^{(-r)}\to A^{(-r-1)}$ given by $a\mapsto a^p$. If $\cO(G)$ is the coordinate ring of $G$, then we write $G^{(r)}$ for the affine group scheme with coordinate ring $\cO(G)^{(r)}$. Recall the notion of the perfection of an algebraic group from \cite{CW}, defined as follows:
\begin{align*}
\cO(G)_{\perf}&=\varinjlim(\cO(G)\xrightarrow{\Fr}\cO(G)^{(-1)}\xrightarrow{\Fr}\cO(G)^{(-2)}\xrightarrow{\Fr}\cdots),\\
G_{\perf}&=\varprojlim(\cdots\xrightarrow{\Fr^*}G^{(-2)}\xrightarrow{\Fr^*}G^{(-1)}\xrightarrow{\Fr^*}G)\cong\Spec(\cO(G)_{\perf}).
\end{align*}
Given a $G$-module $X$ and $r\in\mZ$, $X^{(-r)}$ has a canonical $G^{(-r)}$-action which pulls back to an action of $G_\perf$. As described in \cite[\S 3.3.2]{CW}, every object in $\Rep(G_\perf)$ is of the form $X^{(-r)}$ for some $X$ and $r\geq0$, that is $\Rep(G_\perf)$ is the union of the categories $\Rep(G^{(-r)})$ for $r\geq0$ under the inclusions $\Rep(G^{(-r)})\to\Rep(G^{(-r-1)})$ given by restriction along $\Fr^*$.

For a full subcategory or class of objects $\cM$ in $\Rep G$, write $\cM^{(-r)}$ for the full subcategory or class of objects in $\Rep(G^{(-r)})$ with objects $X^{(-r)}$ for $X\in\cM$. We define a full subcategory $\overline\cT_\infty(G)$ of $\Rep(G_{\perf})$ and thick tensor ideal $\overline I_\infty(G)$ in $\overline\cT_\infty(G)$ as follows:
\[\overline\cT_\infty(G)=\bigcup_{n=1}^\infty\overline\cT_n^{(-n+1)}(G),\quad\overline I_\infty(G)=\bigcup_{n=1}^\infty\overline I_n^{(-n+1)}(G).\]
If $\overline\cI_\infty(G)=\overline I_\infty(G)^{\rm max}$, then we also have $\overline\cI_\infty(G)=\bigcup_{n=1}^\infty\overline\cI_n^{(-n+1)}(G)$ by construction. The following proposition is property (\ref{ItemPerf}) in Theorem~\ref{ThmMain}:

\begin{prop}\label{PropVerInfty}
The abelian envelope of $\overline\cT_\infty(G)/\overline\cI_\infty(G)$ is the union
\[\Ver_{p^\infty}(G)\coloneqq\bigcup_{n=1}^{\infty}\Ver_{p^n}(G^{(-n+1)})\simeq\bigcup_{n=1}^{\infty}\Ver_{p^n}(G)\]
where $\Ver_{p^n}(G)$ is identified with $\Ver_{p^n}(G^{(-n+1)})$ using a split $\mF_p$-rational structure on $G$.
\end{prop}

\begin{proof}
We have the following commutative diagram of categories and functors:
\[\begin{tikzcd}
\Rep G \arrow{r} & \Rep(G^{(-1)}) \arrow{r} & \Rep(G^{(-2)}) \arrow{r} & \cdots \arrow{r}{\text{union}} & \Rep(G_{\perf})\\
\overline\cT_1(G) \arrow{r} \arrow{u} \arrow{d} & \overline\cT_2(G^{(-1)}) \arrow{r} \arrow{u} \arrow{d} & \overline\cT_3(G^{(-2)}) \arrow{r} \arrow{u} \arrow{d} & \cdots \arrow{r}{\text{union}} & \overline\cT_\infty(G) \arrow{u} \arrow{d}\\
\overline\cT_1(G)/\overline\cI_1 \arrow{r} \arrow{d} & \overline\cT_2(G^{(-1)})/\overline\cI_2 \arrow{r} \arrow{d} & \overline\cT_3(G^{(-2)})/\overline\cI_3 \arrow{r} \arrow{d} & \cdots \arrow{r}{\text{union}} & \overline\cT_\infty(G)/\overline\cI_\infty \arrow{d}\\
\Ver_p \arrow{r}(G) & \Ver_{p^2}(G^{(-1)}) \arrow{r} & \Ver_{p^3}(G^{(-2)}) \arrow{r} & \cdots \arrow{r}{\text{union}} & \Ver_{p^\infty}(G)
\end{tikzcd}\]
Suppose there is a pseudo-tensor functor $F:\overline\cT_\infty/\overline\cI_\infty\to\cC$ for some tensor category $\cC$. For each $n$ the functor $\overline\cT_n/\overline\cI_n\to\overline\cT_\infty/\overline\cI_\infty\to\cC$ must factor through $\Ver_{p^n}(G)$, and thus $F$ factors through $\Ver_{p^\infty}(G)$ via a faithful $\bk$-linear monoidal functor $\Ver_{p^\infty}(G)\to\cC$ which is a tensor functor by \cite[Theorem~2.4.1]{quoprop}. Thus $\Ver_{p^\infty}(G)$ is the abelian envelope of $\overline\cT_\infty/\overline\cI_\infty$.
\end{proof}

\begin{remark}\label{RemFrob}
We do not need all of $\overline\cT_n$ to perform these constructions. We could define $\hat\cT_n$ to be the full Karoubi subcategory of $\Rep G$ with indecomposable objects $T(\lambda)^{(m)}$ such that $0\leq m<n$ and $\lambda\in\{0\}\cup((p^{n-m-1}-1)\rho+\Lambda^+)$, and define $\hat I_n$ to be the collection of objects with summands $T(\lambda)^{(m)}$ such that $T(\lambda)\in I_{n-m}$. It then follows from Donkin's tensor product theorem that $\hat\cT_n$ is monoidal and $\hat I_n$ is a thick tensor ideal in $\hat\cT_n$. Let $\hat J_n$ be the thick ideal in $\hat\cT_n$ with indecomposables in $J_n\cup\hat I_n$, so that $\hat J_n$ is the minimal thick ideal strictly containing $\hat I_n$. If we define $\hat\cI_n=\hat I_n^{\rm max}$, then the category $\hat\cT_n/\hat\cI_n$ has abelian envelope $\Ver_{p^n}(G)$ by the same methods as in Theorem~\ref{ThmFullEnv}, and all of the theorems and propositions in this section still hold after replacing $\overline\cT_n,\overline I_n,\overline J_n,\overline\cI_n$ with $\hat\cT_n,\hat I_n,\hat J_n,\hat\cI_n$. In particular, there is a Frobenius functor $\hat\cT_n\to\hat\cT_{n+1}$, and one can define $\hat\cT_\infty$ and $\hat\cI_\infty$ similarly to $\overline\cT_\infty$ and $\overline\cI_\infty$. However, this approach gives a weaker statement for Proposition~\ref{PropExactSeq} below.
\end{remark}

Next we prove property (\ref{ItemExact}) in Theorem~\ref{ThmMain}.

\begin{prop}\label{PropExactSeq}
If $\Sigma:0\to X_1\to\cdots\to X_m\to0$ with $m\in\mN$ is a bounded exact sequence in $\Rep G$ such that $X_i\in\overline\cT_n$ for all $1\leq i\leq m$, then its image in $\Ver_{p^n}(G)$ is also exact.
\end{prop}

\begin{proof}
For $S\in J_n\setminus I_n$, the sequence $\Sigma\otimes S:0\to X_1\otimes S\to\cdots\to X_m\otimes S\to0$ is a sequence of tilting modules. Since $\Sigma\otimes S$ is exact, it is isomorphic to the zero complex in the derived category $D^b(\Rep G)$. But since $\Rep G$ is a highest weight category, the functor $K^b(\Tilt G)\to D^b(\Rep G)$ is an equivalence (e.g. see \cite[\S 2]{Gruber}), and thus $\Sigma\otimes S$ is null-homotopic. This means $\Sigma\otimes S$ is split exact, and thus the image of $\Sigma\otimes S$ in $\Ver_{p^n}(G)$ is exact by additivity of the functor $\Tilt G\to\Ver_{p^n}(G)$. Since the functor $-\otimes S$ in $\Ver_{p^n}(G)$ is exact and faithful, the image of $\Sigma$ in $\Ver_{p^n}(G)$ must also be exact.
\end{proof}

The following is an application of Proposition~\ref{PropExactSeq}. For an object $X$ in a symmetric tensor category, we define its symmetric and exterior powers following \cite[\S 2.1]{CEN}, with $\Sym^nX$ a quotient of $X^{\otimes n}$ and $\Lambda^nX$ a subobject of $X^{\otimes n}$. In particular, $\Sym^2X$ and $\Lambda^2X$ are the cokernel and image respectively of $1-c\in\End(X^{\otimes2})$, where $c=c_{X,X}$ is the braid morphism.

\begin{prop}\label{PropSymExt}
Let $m\in\mN$ and suppose $X$ is a $G$-module such that the modules
\[X,\quad (\Sym^2X^*)^*,\quad \Sym^rX\text{ for }2\leq r\leq m,\quad \Lambda^rX\text{ for }2\leq r\leq m\]
are all in $\overline\cT_n$. Then $\Sym^r F(X)\cong F(\Sym^r X)$ and $\Lambda^r F(X)\cong F(\Lambda^r X)$ for $r\leq m$, where $F$ is the functor $\overline\cT_n\to\Ver_{p^n}(G)$.
\end{prop}

\begin{proof}
For an object $X$ in a symmetric tensor category $\cC$ and $r\geq3$, we have exact sequences
\begin{align*}
\Sigma_1(\cC,X):&\ 0\to(\Sym^2X^*)^*\to X^{\otimes 2}\xrightarrow{1-c}X^{\otimes 2}\to\Sym^2 X\to0\\
\Sigma_2(\cC,X):&\ 0\to\Lambda^2X\to X^{\otimes 2}\to\Sym^2 X\to0\\
\Sigma_3(\cC,X,r):&\ \Lambda^2X\otimes\Sym^{r-2}X\xrightarrow{f_{X,r}} X\otimes\Sym^{r-1}X\xrightarrow{g_{X,r}}\Sym^rX\to0
\end{align*}
(see the proof of \cite[Lemma~5.2.3]{CEN} for $\Sigma_3$). Taking $\cC=\Rep G$ and $X$ satisfying the conditions in the proposition, $F(\Sigma_1(\Rep G,X))$ and $F(\Sigma_2(\Rep G,X))$ are exact sequences in $\Ver_{p^n}(G)$ by Proposition~\ref{PropExactSeq}. Moreover, $\Sigma_3$ is part of a Koszul complex
\[0\to\Lambda^rX\to\Lambda^{r-1}X\otimes X\to\Lambda^{r-2}X\otimes\Sym^2X\to\cdots\to\Sym^rX\to0\]
which is exact in Tannakian categories such as $\Rep G$, and thus $F(\Sigma_3(\Rep G,X,r))$ is also exact by Proposition~\ref{PropExactSeq}. Since $F$ is braided monoidal, it sends the morphism $1-c$ in $\overline\cT_n$ to $1-c$ in $\Ver_{p^n}(G)$, and hence $F(\Sigma_1(\Rep G,X))\cong\Sigma_1(\Ver_{p^n}(G),F(X))$ as complexes in $\Ver_{p^n}(G)$. This means $\Sym^2 F(X)\cong F(\Sym^2X)$, and moreover $F$ sends the projection $X^{\otimes2}\to\Sym^2X$ to the projection $F(X)^{\otimes 2}\to\Sym^2F(X)$, which allows us to conclude that $F(\Sigma_2(\Rep G,X))\cong\Sigma_2(\Ver_{p^n}(G),F(X))$ and $\Lambda^2 F(X)\cong F(\Lambda^2X)$.

Now we consider the morphisms $f_{X,r}$ and $g_{X,r}$ in $\Sigma_3$. We use induction on $r$ and assume that $F(g_{X,r-1})=g_{F(X),r-1}$, with the base case $r=2$ following from the reasoning above. This means $F(f_{X,r})=f_{F(X),r}$, since $F$ is braided monoidal and $f_{X,r}$ is the composition of $\id_X\otimes g_{X,r-1}$ with the inclusion $\Lambda^2X\hookrightarrow X^{\otimes2}$ tensored with $\Sym^{r-2}X$. This allows us to conclude that $F(\Sigma_3(\Rep G,X,r))\cong\Sigma_3(\Ver_{p^n}(G),F(X),r)$, and hence $\Sym^r F(X)\cong F(\Sym^r X)$. To complete the induction and prove that $F(g_{X,r})=g_{F(X),r}$, note that $g_{X,r}$ is the unique morphism whose composition with $X^{\otimes r}\twoheadrightarrow X\otimes\Sym^{r-1}X$ is the projection $X^{\otimes r}\twoheadrightarrow\Sym^rX$, by the universal property of $\Sym^{r-1}X$. $F(g_{X,r})$ satisfies this property for $F(X)$ and hence must equal $g_{F(X),r}$. A dual argument (swapping $\Sym$ with $\Lambda$ and reversing the arrows in $\Sigma_3$) gives $\Lambda^r F(X)\cong F(\Lambda^r X)$.
\end{proof}

\begin{remark}
An immediate application of Proposition~\ref{PropSymExt} is a faster proof of \cite[Corollary~6.4]{cohomology} which states that $\Sym^rL_1=0$ in $\Ver_{p^n}$ if and only if $r\geq p^n-1$. In particular, in $\Rep\SL_2$ we have $\Lambda^2L_1=\unit$ and $\Sym^rL_1=\Delta_r^*\in\overline\cT_n$ for $0\leq r\leq p^n-2$, and then $\Sym^{p^n-1}L_1=\St_n\in I_n(\SL_2)$. Another application is a derivation of the symmetric and exterior powers of simples in $\Ver_9$, confirming the computational results in \cite[\S 5.1.2]{CEN}. It can be computed that both $\Lambda^rL_2$ and $\Sym^rL_2$ are in $\overline\cT_2(\SL_2)$ in characteristic 3 for $0\leq r\leq 7$, and we have $\Lambda^4L_2=0$ and $\Sym^7L_2\in\overline I_2(\SL_2)$. The simples $L_0$ and $L_3$ are the images of the invertible objects in $\Ver_3\simeq\sVec$, and then we obtain the symmetric and exterior powers for the other simples using $L_4\cong L_1\otimes L_3$ and $L_5\cong L_2\otimes L_3$ in $\Ver_9$.
\end{remark}

\section{Alternative descriptions of $\Ver_{p^n}(\SL_2)$}\label{SecVerpnSL2}

The construction of $\Ver_{p^n}(G)$ via $\overline\cT_n(G)$ also reveals new properties of $\Ver_{p^n}=\Ver_{p^n}(\SL_2)$. In particular, we will show that $\Ver_{p^n}$ is equivalent to a particular subcategory of $\Rep\SL_2$ as an abelian category, and can also be expressed as a Serre quotient. First we state some properties of tilting modules unique to the case $G=\SL_2$.

\begin{lemma}\label{LemSL2Ideals}
Let $f:X\to Y$ be a morphism of $\SL_2$-modules.
\begin{enumerate}
\item If $X,Y\in\Tilt\SL_2$, then $f\in\cI_n(\SL_2)$ if and only if $f$ is zero when restricted to any pair of indecomposable summands of $X$ and $Y$ that are not in $I_n(\SL_2)$.
\item If $X,Y\in\overline\cT_n(\SL_2)$, then $f\in\overline\cI_n(\SL_2)$ if and only if $f\otimes\id_{\St_{n-1}}$ is zero when restricted to any pair of indecomposable summands of $X\otimes\St_{n-1}$ and $Y\otimes\St_{n-1}$ that are not in $I_n(\SL_2)$.
\end{enumerate}
\end{lemma}

\begin{proof}
See \cite[Theorem~5.3.1]{tensorideals1} or \cite[Proposition~3.5]{BEO}, plus Lemma~\ref{LemMorFactor}.
\end{proof}

\begin{lemma}\label{LemSimpTilt}
For $i\leq p^n-2$ and $S\in J_n(\SL_2)\setminus I_n(\SL_2)$, we have that $L_i\otimes S$ is a tilting module, and $L_i\otimes S\in I_n$ if and only if $i\geq p^{n-1}(p-1)$.
\end{lemma}

\begin{proof}
Since $T_j$ is a summand of $T_{j-p^{n-1}+1}\otimes\St_{n-1}$ for any indecomposable $T_j\in J_n\setminus I_n$, it suffices to consider $S=\St_{n-1}$. Then $L_i\otimes\St_{n-1}$ is a tilting module by \cite[Lemma~4.3.4]{MAE}. If $i<p^{n-1}(p-1)$ then $L_i\otimes\St_{n-1}$ has highest weight less than $p^n-1$ and thus cannot be in $I_n$. Conversely, if $p^{n-1}(p-1)\leq i\leq p^n-2$ then
\[L_i\otimes\St_{n-1}\cong L_{i-p^{n-1}(p-1)}\otimes L_{p-1}^{(n-1)}\otimes\St_{n-1}\cong L_{i-p^{n-1}(p-1)}\otimes\St_n\in I_n.\]
\end{proof}

\begin{remark}\label{RemNonFaithful}
These two lemmas do not have equivalents in other algebraic groups, so it is unclear whether the results below can be extended to general $\Ver_{p^n}(G)$. Even for objects $X,Y\in J_n(G)\setminus I_n(G)$, Lemma~\ref{LemSimpTilt} may fail for some of the simple factors of $X$ or $Y$, and $\Hom_G(X,Y)\to\Hom_{\Ver_{p^n}(G)}(X,Y)$ is not necessarily injective so Lemma~\ref{LemSL2Ideals} fails as well. For example, for $G=\SL_3$ the tilting module $T((p-2)\rho)$ has $\unit$ in its socle, meaning there is a non-trivial morphism
\[f:\unit\to T(2(p-1)\rho)\otimes T((p-2)\rho)^{(1)}=T((p^2-2)\rho)\in I_2.\]
But $T((p^2-2)\rho)$ is a direct summand of $X\otimes X^*$ where $X=T((p^2-p-2)\varpi_1+p\varpi_2)$ is an object in $J_2\setminus I_2$ (in Figure~\ref{FigSL3}, $X$ is in one of the triangular protrusions at the top of the cell $J_2\setminus I_2$), so we have a non-zero morphism
\[X\xrightarrow{f\otimes\id_X} T((p^2-2)\rho)\otimes X\hookrightarrow X\otimes X^*\otimes X\xrightarrow{\id_X\otimes\ev_X}X\]
which is in $\cI_2(\SL_3)$.
\end{remark}

\subsection{}\label{SecABC} We define the following full subcategories of $\Rep\SL_2$:
\begin{enumerate}
\item $\cA_n$ consists of objects whose weights are all strictly less than $p^n-1$;
\item $\cB_n$ consists of objects whose simple composition factors have highest weights $i$ with $(p-1)p^{n-1}\leq i<p^n-1$;
\item $\cC_n$ consists of objects $X\in\cA_n$ with $\Hom(X,B)=0=\Hom(B,X)$ for all $B\in\cB_n$.
\end{enumerate}
$\cB_n$ is a Serre subcategory of $\cA_n$, and $\cA_n$ is a Serre subcategory of $\Rep\SL_2$. Moreover, $\cA_n$ is a subcategory of $\overline\cT_n(\SL_2)$ by Lemma~\ref{LemSimpTilt}, and all indecomposables in $J_n(\SL_2)\setminus I_n(\SL_2)$ are in $\cA_n$. We will see that $\cC_n$ is abelian, but the inclusion $\cC_n\to\Rep\SL_2$ is not exact in general (Remark~\ref{RemNonExact}).

\begin{lemma}\label{LemAnBn}
We have $\Ob(\cB_n)=\overline I_n\cap\Ob(\cA_n)$, and for any morphism $f$ in $\cA_n$ we have $f\in\overline\cI_n$ if and only if the image $\im(f)$ is an object in $\cB_n$.
\end{lemma}

\begin{proof}
Given $X\in\cA_n$ fix a Jordan-H\"older filtration $0=X_0\subseteq X_1\subseteq\cdots\subseteq X_m=X$, meaning a filtration such that $X_i/X_{i-1}$ is simple for all $i$. By Lemma~\ref{LemSimpTilt}, \cite[Corollary~E.2]{Jantzen}, and induction on $i$, the sequence
\[0\to X_{i-1}\otimes\St_{n-1}\to X_i\otimes\St_{n-1}\to(X_i/X_{i-1})\otimes\St_{n-1}\to0\]
is a sequence of tilting modules and hence is split. Thus, the module $X\otimes\St_{n-1}$ splits into a direct sum $\bigoplus_{i=1}^m(X_i/X_{i-1})\otimes\St_{n-1}$. It then follows from the second statement of Lemma~\ref{LemSimpTilt} that $\Ob(\cB_n)=\overline I_n\cap\Ob(\cA_n)$. Now suppose $f:X\to Y$ is a morphism in $\cA_n$, and fix Jordan-H\"older filtrations $X_i,Y_j$ of $X$ and $Y$ so that $X/X_i=\im(f)=Y_j$ for some $i,j$. By the same reasoning as above, $f\otimes\id_{\St_{n-1}}$ is a split morphism, and by Lemma~\ref{LemSL2Ideals} it factors through an object in $I_n$ if and only if $\im(f)\otimes\St_{n-1}\in I_n$ if and only if $\im(f)\in\cB_n$.
\end{proof}

\begin{lemma}\label{LemTopSocle}
Suppose $X\in\cA_n$, and let $F$ be the functor $\cA_n\to\Ver_{p^n}$
\begin{enumerate}
\item For all objects $P\in J_n\cap\Ob(\cA_n)$, the maps $\Hom(P,X)\to\Hom(F(P),F(X))$ and $\Hom(X,P)\to\Hom(F(X),F(P))$ are isomorphisms.
\item $\Hom(X,B)=0$ for all $B\in\cB_n$ if and only if there exists a surjective morphism $P\twoheadrightarrow X$ for some $P\in J_n\cap\Ob(\cA_n)$.
\item $\Hom(B,X)=0$ for any $B\in\cB_n$ if and only if there exists an injective morphism $X\hookrightarrow Q$ for some $Q\in J_n\cap\Ob(\cA_n)$.
\item $X\in\cC_n$ if and only if $X$ is the image of some morphism $P\to Q$ for some objects $P,Q\in J_n\cap\Ob(\cA_n)$.
\end{enumerate}
\end{lemma}

\begin{proof}
Let $h$ be the map $\Hom(P,X)\to\Hom(F(P),F(X))$ for some $P\in J_n\cap\Ob(\cA_n)$. By \cite[Proposition~2.24]{BEO}, $h$ is surjective. Each indecomposable summand of $P$ has the form
\[T_{2(p^{n-1}-1)-a+bp^{n-1}}\cong T_{2(p^{n-1}-1)-a}\otimes L_b^{(n-1)}\]
for some $0\leq a\leq p^{n-1}-1$ and $0\leq b\leq p-2$, and by Lemma~\ref{LemTiltCover} this has simple top $L_{a+bp^{n-1}}\cong L_a\otimes L_b^{(n-1)}$ which is not in $\overline I_n$. So the image of any non-zero morphism $f:P\to X$ contains such a simple composition factor, and by Lemma~\ref{LemAnBn} this is not in $\cB_n$ and $f\not\in\overline\cI_n$. Thus $h$ is injective and hence an isomorphism. A dual argument shows that $\Hom(X,P)\to\Hom(F(X),F(P))$ is an isomorphism as well, proving (1).

We will prove statement (2), and then (3) follows from a dual argument, and (4) follows from (2) and (3). Suppose there exists a surjection $P\twoheadrightarrow X$. For $B\in\cB_n$, we have $B\in\overline I_n$ by Lemma~\ref{LemAnBn}, so we must have $\Hom(X,B)=0$ or else there exists a morphism $P\to B$ contradicting $\Hom(P,B)\cong\Hom(F(P),0)=0$. For the converse, suppose $\Hom(X,B)=0$ for all $B\in\cB_n$, let $f:P'\twoheadrightarrow F(X)$ be a projective cover in $\Ver_{p^n}$, and take $P$ to be a preimage of $P'$. Suppose for a contradiction that the preimage of $f$ in $\cA_n$ is not surjective, and has a non-zero cokernel $g:X\twoheadrightarrow Y$. Since $f$ is an epimorphism and $F(g)f=0$, we have $F(g)=0$, and so $\im(g)=Y\in\cB_n$ contradicting $\Hom(X,B)=0$ for $B\in\cB_n$. Thus the preimage of $f$ is a surjection $P\twoheadrightarrow X$.
\end{proof}

\begin{lemma}\label{LemCleanIso}
For $X,Y\in\cC_n$, the map $\Hom_{\SL_2}(X,Y)\to\Hom_{\Ver_{p^n}}(X,Y)$ is an isomorphism.
\end{lemma}

\begin{proof}
Let $F$ be the composition $\cC_n\to\overline\cT_n\to\Ver_{p^n}$. $F$ is faithful by Lemmas~\ref{LemSL2Ideals} and \ref{LemAnBn}, so we just need to show that $F$ is full. By Lemma~\ref{LemTopSocle} there are objects $P,Q,R,S\in J_n\cap\Ob(\cA_n)$ (which are also in $\cC_n$ by Lemma~\ref{LemTopSocle}(4)) such that we have morphisms $f_1:P\twoheadrightarrow X$, $f_2:R\twoheadrightarrow Y$, $f_3:X\hookrightarrow Q$, $f_4:Y\hookrightarrow S$. By completing these to exact sequences and applying Proposition~\ref{PropExactSeq}, the images of $f_1,f_2$ (respectively $f_3,f_4$) in $\Ver_{p^n}$ are epimorphisms (respectively monomorphisms). By injectivity and projectivity of the images of $P,Q,R,S$, any morphism $f:F(X)\to F(Y)$ induces morphisms $F(P)\to F(R)$ and $F(Q)\to F(S)$, which have preimages in $\Rep\SL_2$. We now show that there is a unique morphism $g:X\to Y$ making the diagram
\[\begin{tikzcd}
P \arrow[two heads]{r} \arrow{d} & X \arrow[hook]{r} \arrow{d}{g} & Q \arrow{d}\\
R \arrow[two heads]{r} & Y \arrow[hook]{r} & S
\end{tikzcd}\]
commute (note that the outer rectangle automatically commutes by faithfulness of $F$). For uniqueness, if there are two such morphisms $g_1$ and $g_2$ then the compositions $P\twoheadrightarrow X\xrightarrow{g_1}Y\hookrightarrow S$ and $P\twoheadrightarrow X\xrightarrow{g_2}Y\hookrightarrow S$ being equal implies $g_1=g_2$. For existence, suppose $A$ and $B$ are the images of $X\hookrightarrow Q\to S$ and $P\to R\twoheadrightarrow Y$ respectively, meaning they are also the images of $P\twoheadrightarrow X\hookrightarrow Q\to S$ and $P\to R\twoheadrightarrow Y\hookrightarrow S$ respectively. These two maps $P\to S$ are equal, so $A\cong B$ and we can take $g$ to be $X\twoheadrightarrow A\xrightarrow{\sim} B\hookrightarrow Y$. By the same argument, there is a unique morphism $F(X)\to F(Y)$ making the image of this diagram under $F$ commute. Both $F(g)$ and $f$ satisfy this, so $F(g)=f$, proving fullness.
\end{proof}

\begin{theorem}\label{ThmEquiv1}
The composition $\cC_n\to\overline\cT_n\to\Ver_{p^n}$ is an equivalence of abelian categories.
\end{theorem}

\begin{proof}
Lemma~\ref{LemCleanIso} gives fully-faithfulness, so we just need to prove essential surjectivity. Let $F$ be the functor $\overline\cT_n\to\Ver_{p^n}$. For $X\in\Ver_{p^n}$, fix a projective cover $P\twoheadrightarrow X$ and injective hull $X\hookrightarrow Q$, and choose $P',Q'\in J_n\cap\Ob(\cA_n)$ with $F(P')=P$ and $F(Q')=Q$. There is a morphism $f:P'\to Q'$ in $\cA_n$ such that $F(f)$ equals $P\twoheadrightarrow X\hookrightarrow Q$. The image $X'\coloneqq\im(f)$ is in $\cC_n$ by Lemma~\ref{LemTopSocle}(4). Similarly to the proof of Lemma~\ref{LemCleanIso}, the images of $P'\twoheadrightarrow X'$ and $X'\hookrightarrow Q'$ in $\Ver_{p^n}$ are an epimorphism and monomorphism respectively by Proposition~\ref{PropExactSeq}. So $F(X')$ is the image of $P\to Q$ and thus isomorphic to $X$.
\end{proof}

\begin{remark}\label{RemNonExact}
$\cC_n$ is not a monoidal category (except for $p^n=2$), however if $X,Y\in\cC_n$ satisfy $X\otimes Y\in\cC_n$ then the image of $X\otimes Y$ in $\Ver_{p^n}$ is the tensor product of the images of $X$ and $Y$. Also note that the inclusion functor $\cC_n\hookrightarrow\Rep\SL_2$ is not exact in general. For example, for $p=3$ and $n=2$ like in Remark~\ref{RemNonFull}, we have an exact sequence $0\to L_4\to T_6\to\Delta_4\to0$ in $\cC_2\simeq\Ver_9$, but this is not exact in $\Rep\SL_2$ since it has homology $L_6$.
\end{remark}

\subsection{Expressing $\cC_n$ and $\Ver_{p^n}$ as a Serre quotient}\label{SecSerreQuo} Recall that if $\cA$ is an abelian category and $\cB\subseteq\cA$ is a Serre subcategory, the Serre quotient $\cA/\cB$ is a category with objects the same as those in $\cA$ and morphism spaces given by
\[\Hom_{\cA/\cB}(X,Y)=\varinjlim\Hom_{\cA}(X',Y/Y').\]
Here, the colimit is of a directed system where $X',Y'$ range over all subobjects of $X,Y$ such that $X/X',Y'\in\cB$, and the maps are $\Hom(X',Y/Y')\to\Hom(X'',Y/Y'')$ whenever we have inclusion and projection morphisms $X''\hookrightarrow X'$ and $Y/Y'\twoheadrightarrow Y/Y''$.

For $X,Y\in\cA$, define $X_\star$ to be the intersection of all $X'\subseteq X$ with $X/X'\in\cB$, and define $Y^\star$ to be the quotient of $Y$ by the sum of all $Y'\subseteq Y$ with $Y'\in\cB$. If all objects in $\cA$ have finite length, then $X_\star$ and $Y^\star$ give a terminal object in the directed system above, and thus $\Hom_{\cA/\cB}(X,Y)\cong\Hom_\cA(X_\star,Y^\star)$. This means $X\cong X^\star\cong X_\star$ in $\cA/\cB$ by applying the Yoneda lemma to $\Hom_{\cA/\cB}(X,-)\cong\Hom_{\cA/\cB}(X_\star,-)$ and $\Hom_{\cA/\cB}(-,X)\cong\Hom_{\cA/\cB}(-,X^\star)$. Moreover, $X_\star^\star=(X_\star)^\star=(X^\star)_\star$ is the image of the morphism $X_\star\hookrightarrow X\twoheadrightarrow X^\star$, and any $f:X_\star\to Y^\star$ induces a unique morphism $f:X_\star^\star\to Y_\star^\star$. Thus, $\cA/\cB$ is equivalent to the full subcategory of $\cA$ with objects $X$ such that $X=X_\star^\star$, and under this equivalence the functor $\cA\to\cA/\cB$ corresponds to $(-)_\star^\star$.

For $X\in\cA$, we have $X=X_\star$ (respectively $X^\star$) if and only if $\Hom(X,B)=0$ (respectively $\Hom(B,X)=0$) for all $B\in\cB$. This means that, for the categories defined in Section~\ref{SecABC}, $\cC_n$ is the subcategory of $\cA_n$ consisting of objects $X$ such that $X=X^\star_\star$. Thus, by the discussion above, $\cA_n/\cB_n\simeq\cC_n$ and we have proven Theorem~\ref{ThmSL2}. We now show in Theorem~\ref{ThmEquiv2} that the resulting equivalence $\cA_n/\cB_n\simeq\Ver_{p^n}$ can be described more explicitly.

\begin{lemma}\label{LemHomIso}
Let $X,X',T\in\cA_n$ with $X'\subseteq X$ and $T\in J_n$.
\begin{enumerate}
\item If $X'\in\cB_n$ then composition with $X\twoheadrightarrow X/X'$ gives isomorphisms
\[\Hom(T,X)\cong\Hom(T,X/X')\text{ and }\Hom(X/X',T)\cong\Hom(X,T).\]
\item If $X/X'\in\cB_n$ then composition with $X'\hookrightarrow X$ gives isomorphisms
\[\Hom(T,X')\cong\Hom(T,X)\text{ and }\Hom(X,T)\cong\Hom(X',T).\]
\end{enumerate}
\end{lemma}

\begin{proof}
We prove statement (1), and (2) follows from a dual argument. We have long exact sequences
\begin{align*}
&0\to\Hom(X/X',T)\to\Hom(X,T)\to\Hom(X',T)\\
&0\to\Hom(T,X')\to\Hom(T,X)\to\Hom(T,X/X')\to\Ext^1(T,X')
\end{align*}
so it suffices to show $\Hom(X',T)=\Hom(T,X')=\Ext^1(T,X')=0$. If $X'\in\cB_n$ then $T^*\otimes X'\in I_n$ by Lemma~\ref{LemSimpTilt}, so the summands of $T^*\otimes X'$ have highest weights between $p^n-1$ and $2p^n-4$. Thus by \cite[Lemma~5.3.3]{tensorideals1} we have $\Hom(X',T)\cong\Hom(T^*\otimes X',\unit)=0$ and $\Hom(T,X')\cong\Hom(\unit,T^*\otimes X')\cong0$. Since $\unit$ is a tilting module and tilting modules have no extensions by \cite[Corollary~E.2]{Jantzen}, we also have $\Ext^1(T,X')\cong\Ext^1(\unit,T^*\otimes X')=0$.
\end{proof}

\begin{theorem}\label{ThmEquiv2}
We have equivalences $\cC_n\simeq\cA_n/\cB_n\simeq\Ver_{p^n}$ such that the following diagram of functors commutes:
\[\begin{tikzcd} \cC_n \arrow[hook]{r} \arrow[swap]{dr}{\sim} & \cA_n \arrow[hook]{r} \arrow[two heads]{d} & \overline\cT_n \arrow{d} \\
& \cA_n/\cB_n \arrow{r}{\sim} & \Ver_{p^n} \end{tikzcd}\]
\end{theorem}

\begin{proof}
Let $P=T_{p^{n-1}-1}\oplus T_{p^{n-1}}\oplus\cdots\oplus T_{p^n-2}$, so that the image of an object $X\in\cA_n$ in $\Ver_{p^n}=\cC(\cT_n/\cI_n,J_n/I_n)$ is $\Hom_{\overline\cT_n/\overline\cI_n}(P,X)$ as described in Section~\ref{SecAbEnv}. Note that this is the same as $\Hom_{\SL_2}(P,X)$ by Lemma~\ref{LemTopSocle}(1). We define a functor $\Hom(P,-):\cA_n/\cB_n\to\Ver_{p^n}$ sending a morphism $f\in\Hom_{\cA_n/\cB_n}(X,Y)$ to the composition
\[\Hom(P,X)\xrightarrow{\sim}\Hom(P,X')\xrightarrow{f'\circ-}\Hom(P,Y/Y')\xrightarrow{\sim}\Hom(P,Y)\]
for some representative $f':X'\to Y/Y'$ as in Section~\ref{SecSerreQuo}, with the isomorphisms coming from Lemma~\ref{LemHomIso}. To show this is independent of the choice of $f'$, suppose $X''\subseteq X'$ and $Y''\supseteq Y'$, and write $f''$ for the image of $f'$ under the map $\Hom(X',Y/Y')\to\Hom(X'',Y/Y'')$. Since $X'/X''$ is a submodule of $X/X''$, it is in $\cB_n$, and we have a dual result for $Y'$ and $Y''$. Thus we have an isomorphism $\Hom(P,X')\cong\Hom(P,X'')$ by Lemma~\ref{LemHomIso}, meaning $f'$ and $f''$ give the same value for $f$ under the functor $\Hom(P,-)$. That this is an equivalence follows from the equivalences $\cC_n\simeq\cA_n/\cB_n$ described in Section~\ref{SecSerreQuo} and $\cC_n\simeq\Ver_{p^n}$ from Theorem~\ref{ThmEquiv1}.
\end{proof}

\end{document}